\pdfoutput=1

\documentclass{siamart0216}

\usepackage{fullpage}
\usepackage{amsmath,amssymb,amsfonts,mathrsfs}
\usepackage[titletoc,toc,title]{appendix}

\usepackage{array} 
\usepackage[utf8]{inputenc}
\usepackage{listings}
\usepackage{mathtools}
\usepackage{pdfpages}
\usepackage[textsize=footnotesize,color=green]{todonotes}
\usepackage{bm}
\usepackage[normalem]{ulem}
\usepackage{hhline}

\usepackage{algorithm}
\usepackage[noend]{algpseudocode}
\usepackage{algorithmicx}
\algblock{ParFor}{EndParFor}
\algnewcommand\algorithmicparfor{\textbf{parfor}}
\algnewcommand\algorithmicpardo{\textbf{do}}
\algnewcommand\algorithmicendparfor{\textbf{end\ parfor}}
\algrenewtext{ParFor}[1]{\algorithmicparfor\ #1\ \algorithmicpardo}
\algrenewtext{EndParFor}{\algorithmicendparfor}

\usepackage{graphicx}
\usepackage{subfig}
\usepackage{color}

\usepackage{pgfplots}
\usepackage{pgfplotstable}
\definecolor{markercolor}{RGB}{124.9, 255, 160.65}
\pgfplotsset{width=10cm,compat=1.3}
\pgfplotsset{
tick label style={font=\small},
label style={font=\small},
legend style={font=\small}
}

\usetikzlibrary{calc}

\newcommand{\td}[2]{\frac{{\rm d}#1}{{\rm d}{\rm #2}}}
\newcommand{\pd}[2]{\frac{\partial#1}{\partial#2}}

\newcommand{\mbb}[1]{\mathbb{#1}}

\newcommand{\nor}[1]{\left\| #1 \right\|}

\newcommand{\LRp}[1]{\left( #1 \right)}
\newcommand{\LRs}[1]{\left[ #1 \right]}

\newcommand{\LRb}[1]{\left| #1 \right|}
\newcommand{\LRc}[1]{\left\{ #1 \right\}}

\newcommand{\Grad} {\ensuremath{\nabla}}
\newcommand{\Div} {\ensuremath{\nabla\cdot}}

\newcommand{\jump}[1] {\ensuremath{\LRs{\![#1]\!}}}
\newcommand{\avg}[1] {\ensuremath{\LRc{\!\{#1\}\!}}}

\newcommand{\eval}[2][\right]{\relax
  \ifx#1\right\relax \left.\fi#2#1\rvert}

\newcolumntype{C}[1]{>{\centering\let\newline\\\arraybackslash\hspace{0pt}}m{#1}}

\newcommand*\diff[1]{\mathop{}\!{\mathrm{d}#1}}

\makeatletter
\renewcommand\d[1]{\mspace{6mu}\mathrm{d}#1\@ifnextchar\d{\mspace{-3mu}}{}}
\makeatother

\author{Jesse Chan\thanks{Department of Mathematics, Virginia Tech, 225 Stanger Street, Blacksburg, VA 24061-0123} \and T. Warburton\footnotemark[1]}

\date{}

\title{GPU-accelerated Bernstein-Bezier discontinuous Galerkin methods for wave problems}

\begin{document}

\maketitle
\begin{abstract}
We evaluate the computational performance of the Bernstein-Bezier basis for discontinuous Galerkin (DG) discretizations and show how to exploit properties of derivative and lift operators specific to Bernstein polynomials for an optimal complexity quadrature-free evaluation of the DG formulation.  Issues of efficiency and numerical stability are discussed in the context of a model wave propagation problem.  We compare the performance of Bernstein-Bezier kernels to both a straightforward and a block-partitioned implementation of nodal DG kernels in a time-explicit GPU-accelerated DG solver.  Computational experiments confirm the advantage of Bernstein-Bezier DG kernels over both straightforward and block-partitioned nodal DG kernels at high orders of approximation.  
\end{abstract}

\section{Introduction}

%
%

In early work, Kl\"{o}ckner et al. \cite{klockner2009nodal} described how to implement high-order Lagrange discontinuous Galerkin (DG) finite element discretizations \cite{hesthaven2007nodal} of Maxwell's equations on graphics processing units (GPUs). 
In this paper, we examine how the choice of high-order Lagrange elements necessitates operations with block dense matrices that create computational bottlenecks, limiting computational throughput on GPUs. We address this issue by investigating on the one hand block-partitioned implementations of kernels involving dense matrix multiplication and on the other hand using  a Bernstein-Bezier basis, that induces blocked discretization matrices with sparse non-zero blocks.  The former is motivated by GPU memory latency \cite{fatahalian2004understanding, volkov2008benchmarking} and strategies adopted by optimized linear algebra libraries \cite{nvidia-cublas}, while the latter is motivated by recent work on efficient algorithms for high order finite element discretizations.  


Bernstein polynomials have been known for over a century (see \cite{farouki2012bernstein} for a survey and retrospective), and are widely used in graphics rendering and computer aided design.  
However, the study of Bernstein-Bezier polynomials has only recently become an active area of research for high-order finite element discretizations.  Both Ainsworth et al.\ \cite{ainsworth2011bernstein,ainsworth2014pyramid} and Kirby and Thinh \cite{kirby2011fast,kirby2012fast} utilized properties of Bernstein-Bezier polynomials to reduce the operational complexity of stiffness matrix assembly and evaluation of stiffness matrix-vector products for high-order continuous finite element discretizations.  Kirby later introduced efficient algorithms for adopting Bernstein-Bezier polynomials in discontinuous Galerkin time-domain discretizations of first order symmetric hyperbolic PDEs \cite{kirby2015efficient}.  These algorithms have complexity that is asymptotically optimal in the number of degrees of freedom.  We extend that work by exposing additional intrinsic structure of Bernstein-Bezier derivative and lift operators to enable an asymptotically optimal implementation of the former and a reduced complexity implementation for the latter. 

A number of bases that yield block sparse finite element matrices have been proposed in the literature. Warburton et al.\ \cite{warburton1999basis} observed that the stiffness matrix for $C^0$ basis functions defined using Jacobi polynomials with a specific choice of parameters yielded very sparse and banded blocks in the discretization of the Laplacian. Similarly, Beuchler et al.\ \cite{beuchler2006new, beuchler2012sparsity} proposed using integrated Jacobi polynomials to construct topological basis functions that yield sparse and nearly diagonal stiffness matrices on simplicial meshes.  These bases rely on orthogonality properties to sparsify global discretization matrices; in contrast, we leverage properties of Bernstein-Bezier basis functions to simultaneously sparsify elemental derivative operators and factorize elemental lift operators to enable GPU friendly implementations.  

The paper is structured as follows - we review the formulation of DG methods for the first order wave equation in Section~\ref{sec:formulation}.  In Section~\ref{sec:implementation}, we describe the discrete formulation, including discussions of polynomial bases and algorithms for GPU architectures.  In Section \ref{sec:results}, we quantify sensitivity of Bernstein-Bezier DG discretizations to finite precision effects and compare the computational performance of different choices of basis and styles of implementation.  Proofs for the properties of different operators under a Bernstein-Bezier basis are included in Appendix~\ref{app:app}.

\section{Variational formulation}
\label{sec:formulation}

We begin with the first order form of the wave equation
\begin{align*}
\frac{1}{\kappa}\pd{p}{\tau}{} + \Div \bm{u} &= f\\
\rho\pd{\bm{u}}{\tau}{} + \Grad p &= 0,
\end{align*}
where $p$ is acoustic pressure, $\bm{u} \in \mbb{R}^3$ is the vector of velocities, and $\rho$ and $\kappa$ are density and bulk modulus, respectively.  These equations hold over the domain $\Omega$ and time $\tau\in[0,T)$.  To simplify presentation, we assume homogeneous forcing $f = 0$ and Dirichlet boundary conditions $p = 0$ on the domain boundary $\partial \Omega$.  We assume also that a triangulation of the domain $\Omega_h = {\cup}_{k=1}^{K} D^k$ with $K$ elements is given, where $D^k$ is a planar tetrahedron with outward facing normal $\bm{n}$, and that $\kappa$ and $\rho$ are constant over each element. We seek a discrete solution $p \in V_h$, where $V_h$ is the space of piecewise polynomials of total degree $N$ over each element $D^k$
\[
V_h = \LRc{p\in L^2(\Omega), \quad \left.p\right|_{D^k} \in P^N(D^k)}, \qquad P^N(D^k) = \LRc{x^iy^jz^k, \quad (x,y,z) \in D^k, \quad i+j+k\leq N}.
\]

We discretize the wave equation in space using a discontinuous Galerkin formulation with upwind fluxes \cite{hesthaven2007nodal}. We define scalar and vector-valued jumps over the faces of $\partial D^k$ 
\begin{align*}
\jump{p} &= p^+ - p^-, \qquad \jump{\bm{u}}_i = \bm{u}_i^+ - \bm{u}_i^-
\end{align*}
where $p^+$ and $p^-$ are the neighboring and local traces of the solution over a face, respectively.  The local semi-discrete form over each element $D^k$ is
\begin{align}
\int_{D^k} \frac{1}{\kappa}\pd{p}{\tau}{}\phi \diff x &= -\int_{D^k} \Div\bm{u}\phi \diff x + \int_{\partial {D^k}} \frac{1}{2}\LRp{\tau_p\jump{p} - \bm{n}^-\cdot \jump{\bm{u}} }\phi \diff x  \\
\int_{D^k} \rho\pd{\bm{u}}{\tau}{}\bm{\psi} \diff x &= - \int_{D^k} \Grad p \cdot  \bm{\psi}^- \diff x + \int_{\partial {D^k}} \frac{1}{2}\LRp{\tau_u \jump{\bm{u}}\cdot \bm{n}^- - \jump{p}}\bm{\psi}^-\cdot \bm{n}^- \diff x,
\end{align}
where $\tau_p = 1/\avg{\rho c}$, $\tau_u = \avg{\rho c}$, and $c^2 = \kappa/\rho$ is the speed of sound.  
The discrete solution is determined by requiring the above to hold for all $v\in V_h$ and all elements $D^k$. 

To impose homogeneous Dirichlet boundary conditions, we enforce the mirror conditions at faces coinciding with the domain boundary
\[
\left.p^+\right|_{\partial \Omega} = -\left.p^-\right|_{\partial \Omega}, \qquad \left.\bm{u}^+\right|_{\partial \Omega} = \left.\bm{u}^-\right|_{\partial \Omega}.
\]
Then, it is straightforward to show that, by choosing $\phi = p$ and $\bm{\psi} = \bm{u}$, that the above formulation is energy stable, such that
\[
\td{}{\tau}\sum_{{D^k} \in \Omega_h} {\int_{D^k} \LRp{ \frac{1}{\kappa} p^2 + \rho \bm{u}^2}}  \leq -\frac{1}{2}\sum_{{D^k} \in \Omega_h} \int_{\partial {D^k}} \LRp{\tau_p\jump{p}^2 + \tau_u\jump{\bm{u}_n}^2}.
\]

\section{Discrete operators and numerical implementation}
\label{sec:implementation}

Assuming that $D^k$ is a planar tetrahedron, each element is an affine mapping of a reference tetrahedron $\widehat{D}$, such that for an elemental mapping $\Phi^k$, 
\[
\LRp{x,y,z} = \Phi^k\LRp{r,s,t}, \qquad \LRp{x,y,z}\in D^k, \qquad \LRp{r,s,t} \in \widehat{D} = \LRc{-1\leq r,s,t, \quad r + s + t \leq 1}.  
\]  
Derivatives are then computed through the chain rule; introducing the matrix of geometric factors $G^k$, the physical gradient $\Grad p$ on $D^k$ is computed as
\[
G^k = \begin{bmatrix}
r_x & s_x & t_x\\
r_y & s_y & t_y\\
r_z & s_z & t_z
\end{bmatrix}, 
\qquad \Grad p = G^k \widehat{\Grad}p,
\]
where $\widehat{\Grad}p$ is the derivative of $p$ with respect to the reference coordinates.  

Assuming a polynomial basis $\LRc{\phi_i}_{i=1}^{N_p}$, we define the reference mass matrix $\bm{M}$ and face mass matrix $\bm{M}^f$ as
\[
\bm{M}_{ij} = \int_{\widehat{D}}{ \phi_j \phi_i}, \qquad \bm{M}^f_{ij} = \int_{f_{\widehat{D}}}  \phi_j \phi_i.
\]
where $f_{\widehat{D}}$ is a face of the reference element $\widehat{D}$.  The reference derivative matrices $\bm{D}^r,\bm{D}^s,\bm{D}^t$ are defined as the linear operators such that
\[
\sum_{j=1}^{N_p}\LRp{\bm{D}^r \bm{p}}_j\phi_j = \pd{p}{r}{}, \qquad \sum_{j=1}^{N_p}\LRp{\bm{D}^s \bm{p}}_j\phi_j = \pd{p}{s}{}, \qquad \sum_{j=1}^{N_p}\LRp{\bm{D}^t \bm{p}}_j\phi_j = \pd{p}{t}{}, \qquad p = \sum_{j=1}^{N_p} \bm{p}_j \phi_j.
\]
where $\bm{p}$ is the vector of degrees of freedom representing a function $p \in V_h$.  

We introduce the determinant of the volume Jacobian $J^k$ and determinant of the face Jacobian $J^f$ as additional geometric factors.  
By exploiting the fact that planar simplices are affinely related, and that geometric factors are constant over each element, the DG formulation may be expressed locally as
\begin{align*}
\frac{1}{\kappa}J^k\bm{M}\td{\bm{p}}{t} &= J^k\bm{M} \sum_{i = 1}^3 \LRp{G^k_{i0} \bm{D}^r + G^k_{i1} \bm{D}^r + G^k_{i2} \bm{D}^t} \bm{U}_j + \sum_{f=1}^{N_{\text{faces}}} J^f\bm{M}^f F_p(\bm{p}^-,\bm{p}^+,\bm{U}^-,\bm{U}^+),\\
\rho J^k\bm{M}\td{\bm{U}_i}{t} &= J^k\bm{M} \LRp{G^k_{i0} \bm{D}^r + G^k_{i1} \bm{D}^r + G^k_{i2} \bm{D}^t}  \bm{p} + \sum_{f=1}^{N_{\text{faces}}} J^f \bm{n}_i\bm{M}^f  F_{u}(\bm{p}^-,\bm{p}^+,\bm{U}^-,\bm{U}^+),
\end{align*}
where $\bm{U}_i$ are degrees of freedom for $\bm{u}_i$, $\bm{p}^-, \bm{p}^+$ are degrees of freedom for $p$ on $D^k$ and its neighbor across a face $f$, respectively, and $F_p,F_u$ are defined such that 
\begin{align*}
\LRp{ J^f\bm{M}^f F_p(\bm{p}^-,\bm{p}^+,\bm{U}^-,\bm{U}^+)}_i &= \int_{f_{D^k}} \frac{1}{2}\LRp{\tau_p \jump{p} - \bm{n}^-\cdot\jump{\bm{u}}}\phi_i^-\\ 
\LRp{J^f \bm{n}_i \bm{M}^f F_u(\bm{p}^-,\bm{p}^+,\bm{U}^-,\bm{U}^+)}_i &= \int_{f_{D^k}} \frac{1}{2}\LRp{\tau_u\jump{\bm{u}} \cdot \bm{n}^- - \jump{p}}\bm{\psi}_i^- \bm{n}_i^-.
\end{align*}
Inverting $\bm{M}$ results in the system of ODEs
\begin{align}
\label{eq:discrete_var}
\frac{1}{\kappa}\td{\bm{p}}{t} &= \sum_{i = 1}^3 \LRp{G^k_{i0} \bm{D}^r + G^k_{i1} \bm{D}^r + G^k_{i2} \bm{D}^t}  \bm{U}_j + \sum_{f=1}^{N_{\text{faces}}} \frac{J^f}{J^k}\bm{L}^f F_p(\bm{p}^-,\bm{p}^+,\bm{U}^-,\bm{U}^+),\\
\rho\td{\bm{U}_i}{t} &=\LRp{G^k_{i0} \bm{D}^r + G^k_{i1} \bm{D}^r + G^k_{i2} \bm{D}^t}  \bm{p} + \sum_{f=1}^{N_{\text{faces}}} \frac{J^f}{J^k}n_i\bm{L}^f F_u(\bm{p}^-,\bm{p}^+,\bm{U}^-,\bm{U}^+), \qquad i = 0,\ldots,d.\notag
\end{align}
where $\bm{L}^f$ is the lift operator for a face $f$
\[
\bm{L}^f = \bm{M}^{-1}\bm{M}^f.
\]
By using the ``strong'' DG formulation, the mass matrix is factored out of the volume terms and removed after multiplying with the inverse mass matrix on both sides.  This leaves only the application of a derivative operator, and decreases the number of matrix multiplications required to evaluate the volume term.  To evaluate the surface term, the inverse of the mass matrix is fused into the face mass matrices to produce the face lift matrices $\bm{L}^f$.  We define $\bm{L}$ as the column-concatenation of these lift operators 
\[
\bm{L} = \LRp{\bm{L}^1 \middle| \ldots \middle| \bm{L}^4}. 
\]
For bases which are rotationally symmetric (including both the Lagrange basis with symmetric point distributions and the Bernstein-Bezier basis), the block-columns of $\bm{L}$ (the lift matrices $\bm{L}^f$ for a face) are identical up to row permuations. 

These derivative and lift operators can be made sparse under and appropriate choice of basis, and the efficient application of such operators motivates the investigation of the Bernstein-Bezier basis as an alternative to nodal bases in the following sections.  



\subsection{Choice of polynomial basis}

In this work, we consider two choices of local bases which span the space of multi-dimensional polynomials of total degree $N$ on the reference bi-unit right tetrahedron with coordinates $(r,s,t)$.  The first basis consists of Lagrange polynomials $\left\{ \ell_{i} \right\}_{i=1}^{i={N_p}}$ associated with a node set $\left\{ r_i, s_i, t_i \right\}_{i=1}^{i={N_p}}$ with dimension 
\[
N_p=\frac{(N+1)(N+2)(N+3)}{6}.
\]
These Lagrange basis functions possess the standard interpolatory Kronecker delta property
\[
\ell_i\left(r_j,s_j,t_j\right) = \delta_{ij}, \qquad 1\leq i,j\leq N_p.
\]
The numerical stability of operations involving a Lagrange basis depends heavily on the choice of nodes.  In this work, we adopt the Warp and Blend nodes \cite{warburton2006explicit}, which are optimized for interpolation quality.  The conditioning of the Vandermonde matrix is also improved under this optimized choice of interpolation nodes.  

By choosing a unisolvent nodal set $\LRp{r_i,s_i,t_i}$ which includes $(N+1)(N+2)/2$ nodes per face and nodes at all vertices, the Lagrange basis may be geometrically decomposed into vertex, edge, face, and interior basis functions.  As DG methods couple neighboring elements together through fluxes, the trace of the solution must be computed on shared faces.  Since these face traces are polynomial, they can be determined by node-to-node connectivity maps and face degrees of freedom under a Lagrange basis.  In comparison, modal (hierarchical and orthogonal) bases require additional computation to evaluate face traces of the solution, or may require orientation-aware procedures to ensure compatibility of traces between neighboring elements. 

The second basis under consideration is the Bernstein-Bezier basis of degree $N$.  On a tetrahedron, Bernstein polynomials are most easily expressed in terms of barycentric coordinates.  For the reference tetrahedron, these coordinates are given in terms of the reference coordinates
\[
\lambda_0 = -\frac{(1+r+s+t)}{2}, \qquad \lambda_1 = \frac{(1+r)}{2}, \qquad \lambda_2  =  \frac{(1+s)}{2}, \qquad \lambda_3  =  \frac{(1+t)}{2}.
\]
The Bernstein polynomials of degree $N$ are then simply scaled monomials in the barycentric coordinates
\[
B^N_{ijkl} = C^N_{ijkl}\lambda^i_0 \lambda^j_1 \lambda^k_2 \lambda^l_3, \qquad C^N_{ijkl} = \frac{N!}{i!j!k!l!}
\]
where the index $l$ is defined to be $l = N-i-j-k$.  The coefficients which define an expansion in terms of Bernstein polynomials are referred to as control points.  

The Bernstein-Bezier basis also admits a geometric decomposition into vertex, edge, face, and interior basis functions.  Additionally, while Bernstein polynomials do not possess the discrete Kronecker delta property of Lagrange polynomials, each basis function attains its maximum at one node in an equispaced lattice on a tetrahedron.  As such, they share many of the convenient properties of Lagrange polynomials as used in nodal DG methods.  For example, nodal discontinuous Galerkin methods determine fluxes using node-to-node connectivity maps and degrees of freedom corresponding to face points on two neighboring elements.  The jumps of polynomial solutions across a face are then determined using the difference between nodal values at face nodes on each element.  Bernstein-Bezier DG implementations may similarly compute the jumps of polynomial solutions using the difference between the Bernstein-Bezier control points on the shared face of two neighboring elements.  The correspondence between Bernstein polynomials and equispaced points on a tetrahedron makes it simple and straightforward to incorporate Bernstein polynomials into a classical nodal DG implementation.  

\subsubsection{Derivative matrices}
\label{sec:dmats}
For both nodal and Bernstein-Bezier bases, derivative matrices have a particular structure.  Figure~\ref{fig:dmats} shows the sparsity pattern of nodal and Bernstein-Bezier derivative matrices, where the Bernstein-Bezier derivative is taken with respect to the barycentric coordinate $\lambda_0$.  Despite the lack of an explicit formula for nodal bases on simplices, the matrix $\bm{D}^r$ for the Lagrange basis shows some sparsity.  This was exploited to generate tuned matrix-multiplication code based on predetermined sparsity patterns by Wozniak et al.\ in \cite{wozniak2014gimmik}.  We take an alternative approach here and explore the Bernstein polynomial basis, which yields sparse and structured derivative operators.  

\begin{figure}
\centering
\subfloat[Nodal derivative matrix ($\bm{D}^r$)]{\includegraphics[width=.3\textwidth]{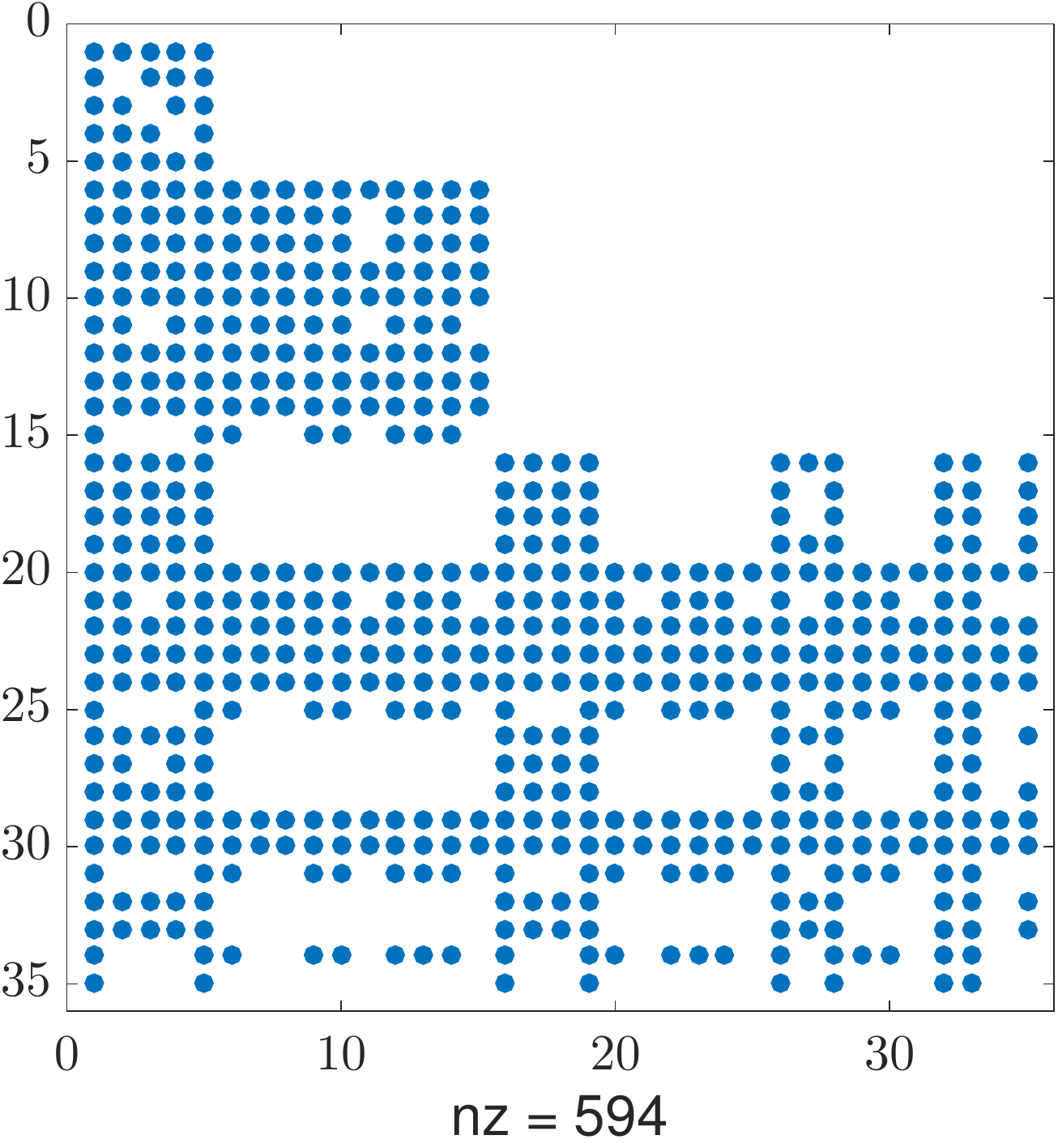}}
\hspace{5em}
\subfloat[Bernstein derivative matrix ($\bm{D}^0$)]{\includegraphics[width=.3\textwidth]{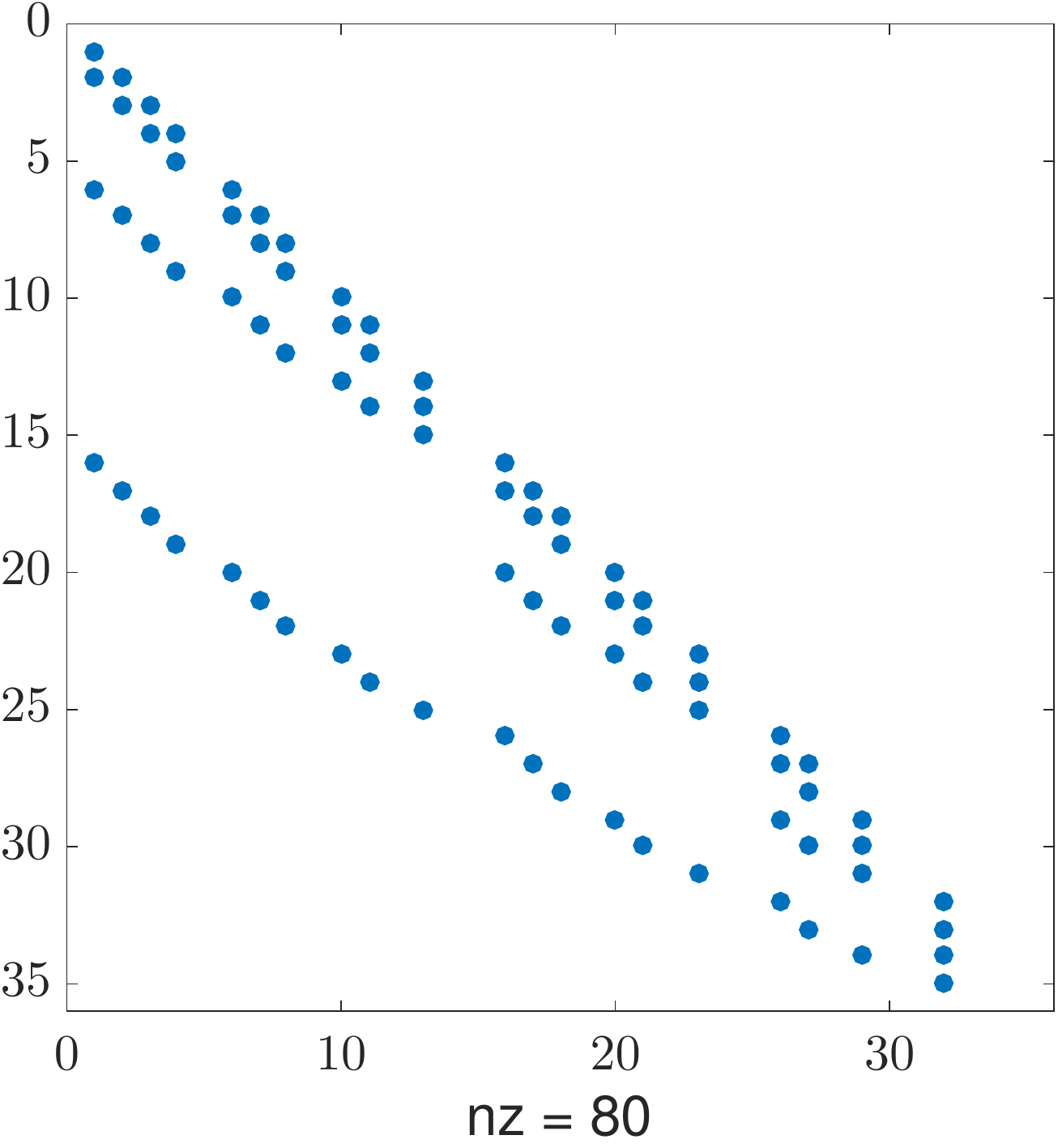}}
\caption{Sparsity patterns of derivative operators using nodal and Bernstein-Bezier bases for $N = 4$.}
\label{fig:dmats}
\end{figure}

A key property of the Bernstein-Bezier basis is the sparsity of the differentiation matrices, inherited from the property that the derivative of a degree $N$ Bernstein polynomial with respect to a barycentric coordinate is representable as a linear combination of no more than $4$ other degree $N$ Bernstein polynomials.\footnote{In general, the derivative of a $d$-dimensional Bernstein polynomial of degree $N$ is representable as a linear combination of at most $(d+1)$ other degree $N$ Bernstein polynomials.}  These combinations are explicitly known, and correspond to the sparsity patterns of differentiation matrices.  In three dimensions, each index corresponds to the tuple of barycentric indices $(i,j,k,l)$.  The sparsity pattern of the derivative matrix with respect to the barycentric coordinate $\lambda_0$ is summarized as follows:  
\begin{lemma}
\label{lemma:lemma1}
The non-zero rows of a given column $(i,j,k,l)$ of $\bm{D}^0$ are
\begin{align*}
\bm{D}^0_{(i,j,k,l),(i,j,k,l)} &= i\\
\bm{D}^0_{(i-1,j+1,k,l),(i,j,k,l)} &= j+1\\
\bm{D}^0_{(i-1,j,k+1,l),(i,j,k,l)} &= k+1\\
\bm{D}^0_{(i-1,j,k,l+1),(i,j,k,l)} &= l+1.
\end{align*}
Similarly, non-zero rows of $\bm{D}^1, \bm{D}^2, \bm{D}^3$ for a column $(i,j,k,l)$ are
\begin{align*}
\bm{D}^1_{(i,j,k,l),(i,j,k,l)} &= j, & \bm{D}^2_{(i,j,k,l),(i,j,k,l)} &= k,  &\bm{D}^3_{(i,j,k,l),(i,j,k,l)} &= l,\\
\bm{D}^1_{(i+1,j-1,k,l),(i,j,k,l)} &= i+1, &\bm{D}^2_{(i+1,j,k-1,l),(i,j,k,l)} &= i+1,  &\bm{D}^3_{(i+1,j,k,l-1),(i,j,k,l)} &= i+1,\\
\bm{D}^1_{(i,j-1,k+1,l),(i,j,k,l)}&= k+1,  &\bm{D}^2_{(i,j+1,k-1,l),(i,j,k,l)} &= j+1,  &\bm{D}^3_{(i,j+1,k,l-1),(i,j,k,l)} &= j+1, \\
\bm{D}^1_{(i,j-1,k,l+1),(i,j,k,l)}&= l+1,  &\bm{D}^2_{(i,j,k+1,l-1),(i,j,k,l)} &= k+1,  &\bm{D}^3_{(i,j,k-1,l+1),(i,j,k,l)} &= l+1.
\end{align*}
The non-zero column indices of $\bm{D}^1,\ldots,\bm{D}^4$ can also be determined explicitly:
\begin{align*}
\bm{D}^0_{(i,j,k,l),(i,j,k,l)} &= i, &\bm{D}^1_{(i,j,k,l),(i,j,k,l)} &= j, & \bm{D}^2_{(i,j,k,l),(i,j,k,l)} &= k,  &\bm{D}^3_{(i,j,k,l),(i,j,k,l)} &= l,\\
\bm{D}^0_{(i,j,k,l),(i+1,j-1,k,l)} &= j, &\bm{D}^1_{(i,j,k,l),(i-1,j+1,k,l)} &= i, &\bm{D}^2_{(i,j,k,l),(i-1,j,k+1,l)} &= i,  &\bm{D}^3_{(i,j,k,l),(i-1,j,k,l+1)} &= i,\\
\bm{D}^0_{(i,j,k,l),(i+1,j,k-1,l)} &= k, &\bm{D}^1_{(i,j,k,l),(i,j+1,k-1,l)}&= k,  &\bm{D}^2_{(i,j,k,l),(i,j-1,k+1,l)} &= j,  &\bm{D}^3_{(i,j,k,l),(i,j-1,k,l+1)} &= j, \\
\bm{D}^0_{(i,j,k,l),(i+1,j,k,l-1)} &= l, &\bm{D}^1_{(i,j,k,l),(i,j+1,k,l-1)}&= l,  &\bm{D}^2_{(i,j,k,l),(i,j,k,l-1)} &= k,  &\bm{D}^3_{(i,j,k,l),(i,j,k-1,l+1)} &= l.
\end{align*}
\end{lemma}
The proof of the above is derived using multi-index notation in Appendix~\ref{app:bb}, and details of their computational implementation are given in Section~\ref{sec:comp}.  These derivative operators are most conveniently stored as a list of $N_p \times 4$ row indices.  Since the values of $\bm{D}^0,\ldots, \bm{D}^3$ are identical, only one additional list of $N_p \times 4$ values is stored and reused for each barycentric derivative.  

\subsubsection{Lift matrices}
\label{sec:lmat}
For nodal bases, the lift matrix is dense, without any evident underlying structure.  However, for the Bernstein-Bezier basis, the lift matrix $\bm{L}$ displays a rich structure that allows for a more efficient evaluation of its matrix-vector product.  To explain this structure, we begin by recalling several useful facts about the Bernstein-Bezier basis.  

We introduce the two-dimensional degree elevation operator $\bm{E}_{N-i}^{N}$, which expresses polynomials of degree $N-i$ on a triangle as degree $N$ polynomials on a triangle (for $i \leq N$).  For Bernstein polynomials, the process of degree elevation between consecutive degrees is very efficient - a two-dimensional degree elevation operator $\bm{E}_{N-1}^N$ contains at most $3$ entries per row irregardless of the degree $N$ \cite{kirby2015efficient}.\footnote{In general, a $(d-1)$ dimensional degree elevation operator $\bm{E}_{N-1}^N$ contains at most $d$ entries per row.}  Degree elevation between arbitrary degrees is similarly achieved by composing one-degree elevation operators
\[
\bm{E}_{N-i}^N = \bm{E}_{N-1}^N \bm{E}^{N-1}_{N-2}\ldots \bm{E}_{N-i}^{N-i+1}.
\]
We note that degree elevation operators may also be used to restrict a face (two-dimensional) mass matrix of degree $N$ to a face mass matrix of lower degree, i.e.\
\[
\bm{M}^{f}_{N-i,N-j} = \LRp{\bm{E}_{N-i}^N}^T \bm{M}_{N,N}^{f} \bm{E}_{N-j}^N.
\]
For this reason, we refer to the transpose of the degree elevation operator as the degree reduction operator.  

We consider the lift matrix for a single face $\bm{L}^f$.  We observe the following structure: let $N^f_{p}$ and $(N-i)^f_p$ be the dimension of the total degree $N$ and $(N-i)$ polynomial spaces on a triangular face, respectively.  For convenience, we assume that the degrees of freedom are ordered such that the first $N^f_p$ are associated with the face $f$, the next $(N-1)^f_p$ are associated with the layer of equispaced nodes in the direction normal to the face, and so on.  Then, there exist scaling constants $\ell_i$ such that lift matrix displays the structure
\[
\bm{L}^f = \LRp{ \begin{array}{c} \bm{L}_0 \\ \vdots \\ \bm{L}_N \end{array}} = \LRp{ \begin{array}{c} \bm{L}_0 \\ \ell_1 \LRp{\bm{E}_{N-1}^N}^T \bm{L}_0 \\ \vdots \\ \ell_N \LRp{\bm{E}_{0}^N}^T \bm{L}_0 \end{array}}, \qquad \bm{L}_i \in \mbb{R}^{(N-i)^f_p \times N^f_p }, 
\]
where $\ell_1,\ldots, \ell_N$ and $\bm{L}_0$ are given explicitly by
\[
\ell_i = (-1)^i \frac{\binom{N}{i}}{1+i}, \qquad \bm{L}_0 = \frac{(N+1)^2}{2} \LRp{\bm{E}^{N+1}_N}^T \bm{E}^{N+1}_N.
\]
Supporting lemmas and additional properties of the Bernstein-Bezier lift matrix are given in Appendix~\ref{app:lift}.

\begin{figure}
\centering
\subfloat[Bernstein lift matrix]{\includegraphics[height=.17\textheight]{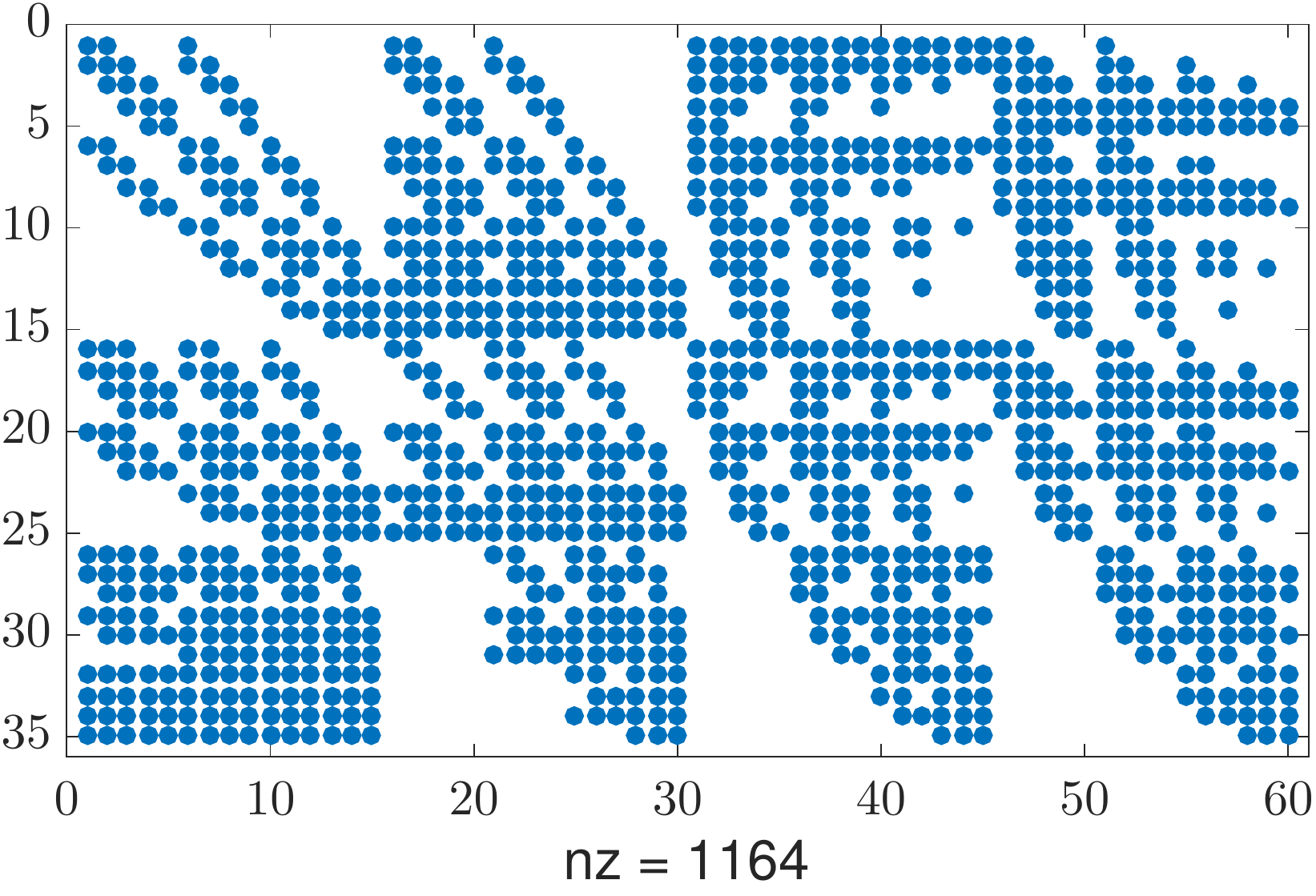}}
\hspace{2em}
\subfloat[$\bm{E}_L$]{\includegraphics[height=.17\textheight]{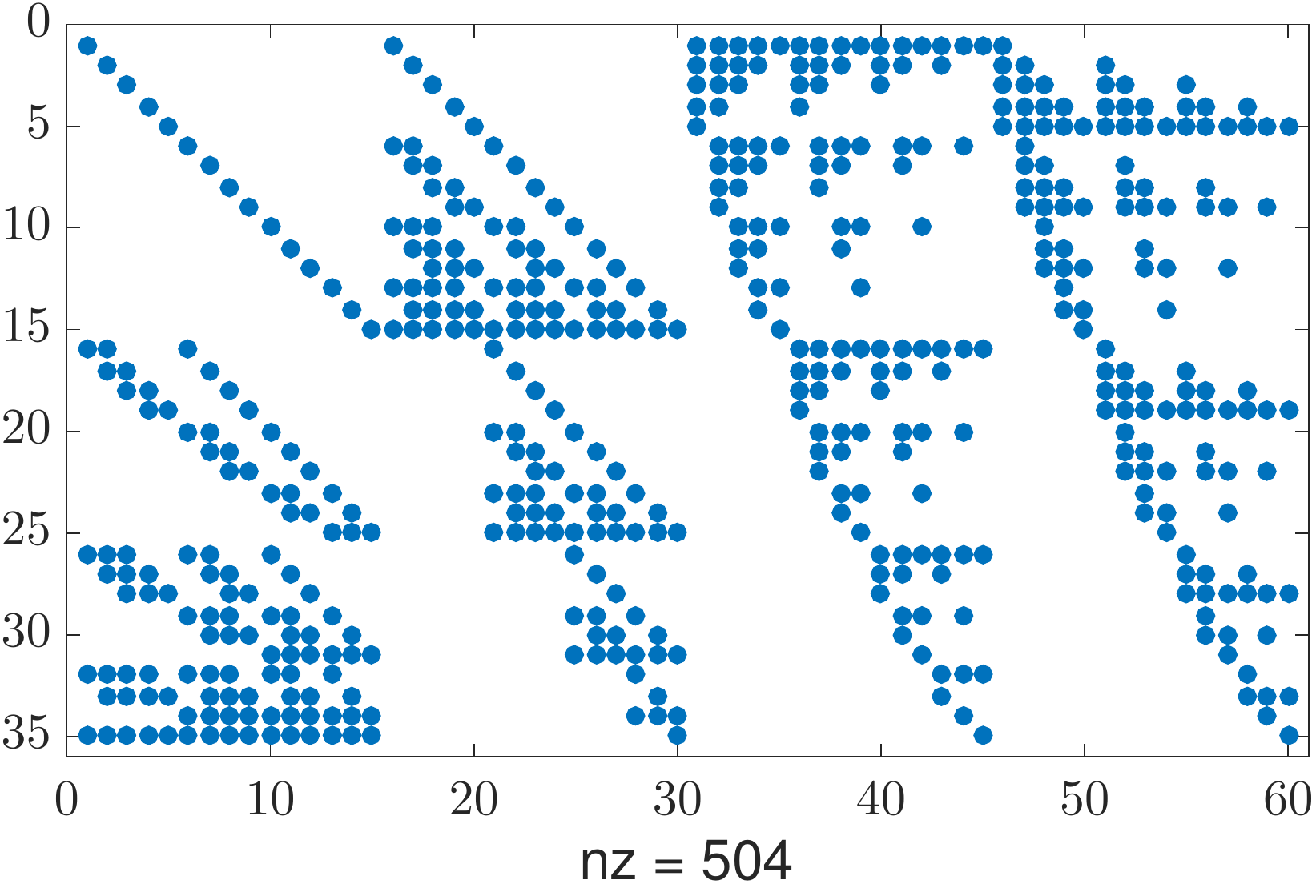}}
\hspace{2.25em}
\subfloat[$\bm{L}_0$]{\includegraphics[height=.17\textheight]{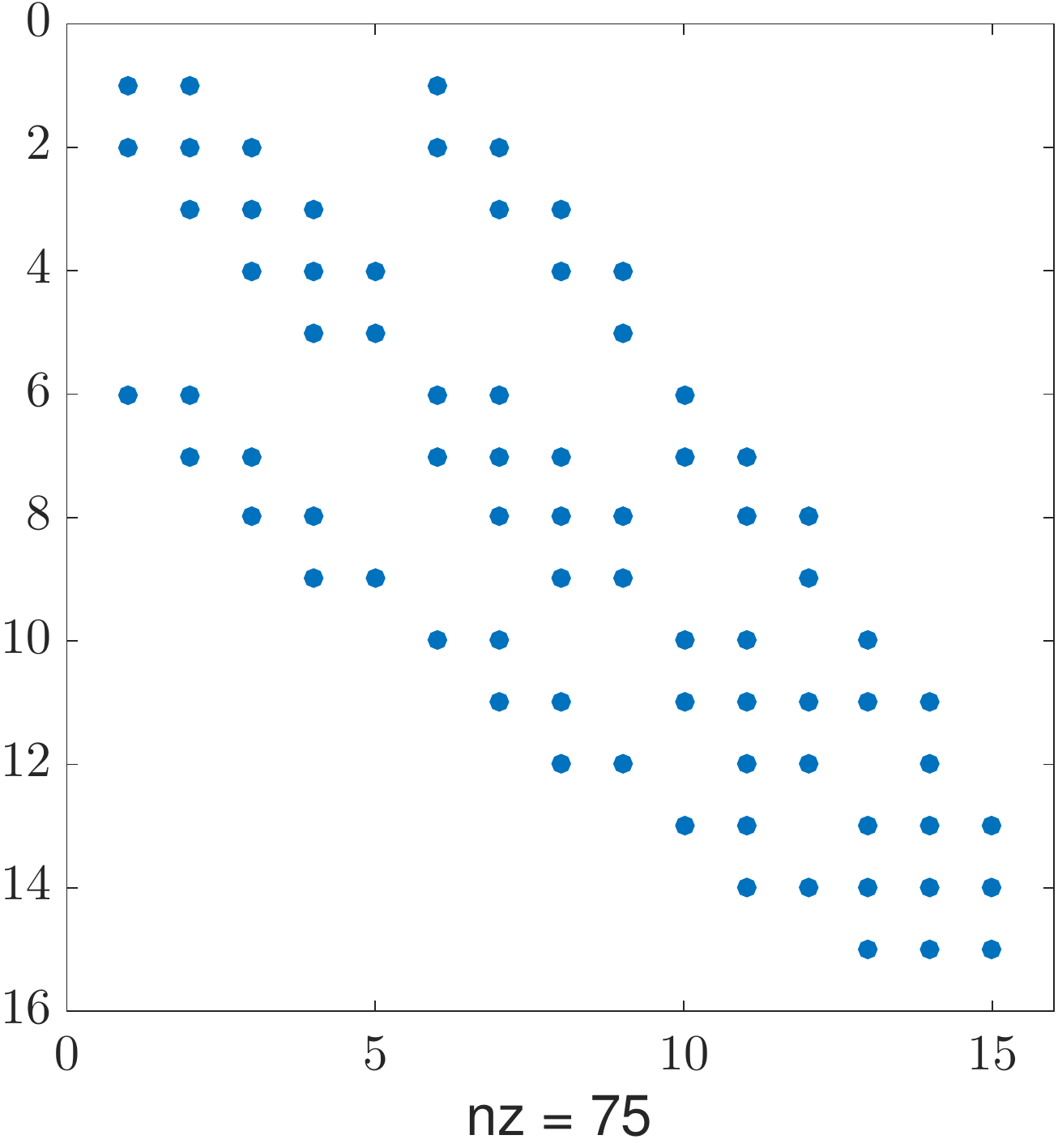}}
\caption{Sparsity patterns of lift operators $\bm{L}$, $\bm{E}_L$, and $\bm{L}_0$ for a Bernstein-Bezier basis of degree $N=4$.}
\label{fig:lmats}
\end{figure}


Let $\bm{p}^f$ be the flux degrees of freedom for a given face.  It is then possible to apply $\bm{L}^f$ to $\bm{p}^f$ using degree reduction and $\bm{L}_0$.  Assuming the degrees of freedom are ordered as described above, we define the face reduction matrix $\bm{E}_L^f$ in terms of degree reduction operators
\[
\bm{E}_L^f = \LRp{\begin{array}{c}
\bm{I}\\ \ell_1 \LRp{\bm{E}_{N-1}^N}^T \\ \ell_2 \LRp{\bm{E}_{N-2}^N}^T\\ \vdots \\ \ell_N \LRp{\bm{E}_0^N}^T
\end{array}}.
\]
Then, the face lift matrix $\bm{L}^f$ may be factorized as 
\[
\bm{L}^f = \bm{E}_L^f \bm{L}_0.
\]
Each degree reduction operator can be decomposed into the product of one-degree reduction operators 
\[
\LRp{\bm{E}_{N-i}^N}^T = \LRp{\bm{E}^{N-i+1}_{N-i}}^T\LRp{\bm{E}^{N-i+2}_{N-i+1}}^T\ldots \LRp{\bm{E}^N_{N-1}}^T.
\] 
Fast algorithms for Bernstein polynomials typically exploit the fact that one-degree reduction operator is sparse, containing at most $d+1$ non-zero entries per row in $d$ dimensions.  To apply the face reduction matrix $\bm{E}_L^f$ requires applying at most $(N+1)$ triangular degree reduction operations to $(N+1)$ sets of triangular face data.  Expanding degree reduction operators as the product of one-degree reduction operators gives
\[
\bm{E}_L^f = \LRp{\begin{array}{c}
\bm{I}\\ 
\ell_1 \LRp{\bm{E}_{N-1}^N}^T \\ 
\ell_2 \LRp{\bm{E}_{N-2}^{N-1}}^T \LRp{\bm{E}_{N-1}^N}^T \\ \vdots \\ 
\ell_N \LRp{\bm{E}_0^1}^T\LRp{\bm{E}_1^2}^T\ldots \LRp{\bm{E}_{N-1}^N}^T
\end{array}}.
\]
The application $\bm{E}_L^f$ to a vector of $(N+1)(N+2)/2$ coefficients $\bm{u}^f$ is described in Algorithm~\ref{alg:slice_lift}.  Figure~\ref{fig:optlift} shows an illustration of this algorithm for applying the lift matrix, which can be interpreted as a slice-by-slice sweep through a simplex.  
\begin{algorithm}
\caption{Optimal-complexity application of a Bernstein-Bezier face lift matrix $\bm{E}_L^f$.} 
\label{alg:slice_lift} 
\begin{algorithmic}[1]
\State Compute and store $\bm{L}_0\bm{u}^f$.  
\State Store the first $(N+1)(N+2)/2$ entries of $\bm{E}_L^f \bm{u}^f$ as $\bm{L}_0\bm{u}^f$.
\State Compute and store $\LRp{\bm{E}_{N-1}^N}^T \bm{u}^f$.
\For{$j = 1,\ldots,N$} 
\State Store the next $(N+1-j)(N+2-j)/2$ entries of $\bm{E}_L^f \bm{u}^f$ as $\ell_j\LRp{\bm{E}_{N-j}^N}^T \bm{u}^f$.
\State Compute and store $\LRp{\bm{E}_{N-j-1}^N}^T \bm{u}^f$ by applying $\LRp{\bm{E}^{N-j}_{N-j-1}}^T$ to $\LRp{\bm{E}_{N-j}^N}^T \bm{u}^f$.
\EndFor
\end{algorithmic}
\end{algorithm}
$\bm{L}_0$ may be applied in a sparse manner as there are no more than $7$ non-zero entries per row (independent of $N$), yielding an asymptotic cost of $O(N^{d-1})$.  Since the main cost of the algorithm is $N$ one-degree reductions of cost $O(N^2)$ each, this results in an asymptotically optimal complexity of $O(N^3) = O(N^d)$ operations, which improves upon the $O(N^4) = O(N^{d+1})$ asymptotic complexity of applying the inverse of the mass matrix using techniques in \cite{kirby2015efficient}.  Additionally, while the implementation of Algorithm~\ref{alg:slice_lift} requires more memory for intermediate storage, it does not increase the asymptotic memory cost.  
\begin{figure}
\centering
\subfloat{\includegraphics[width=.32\textwidth]{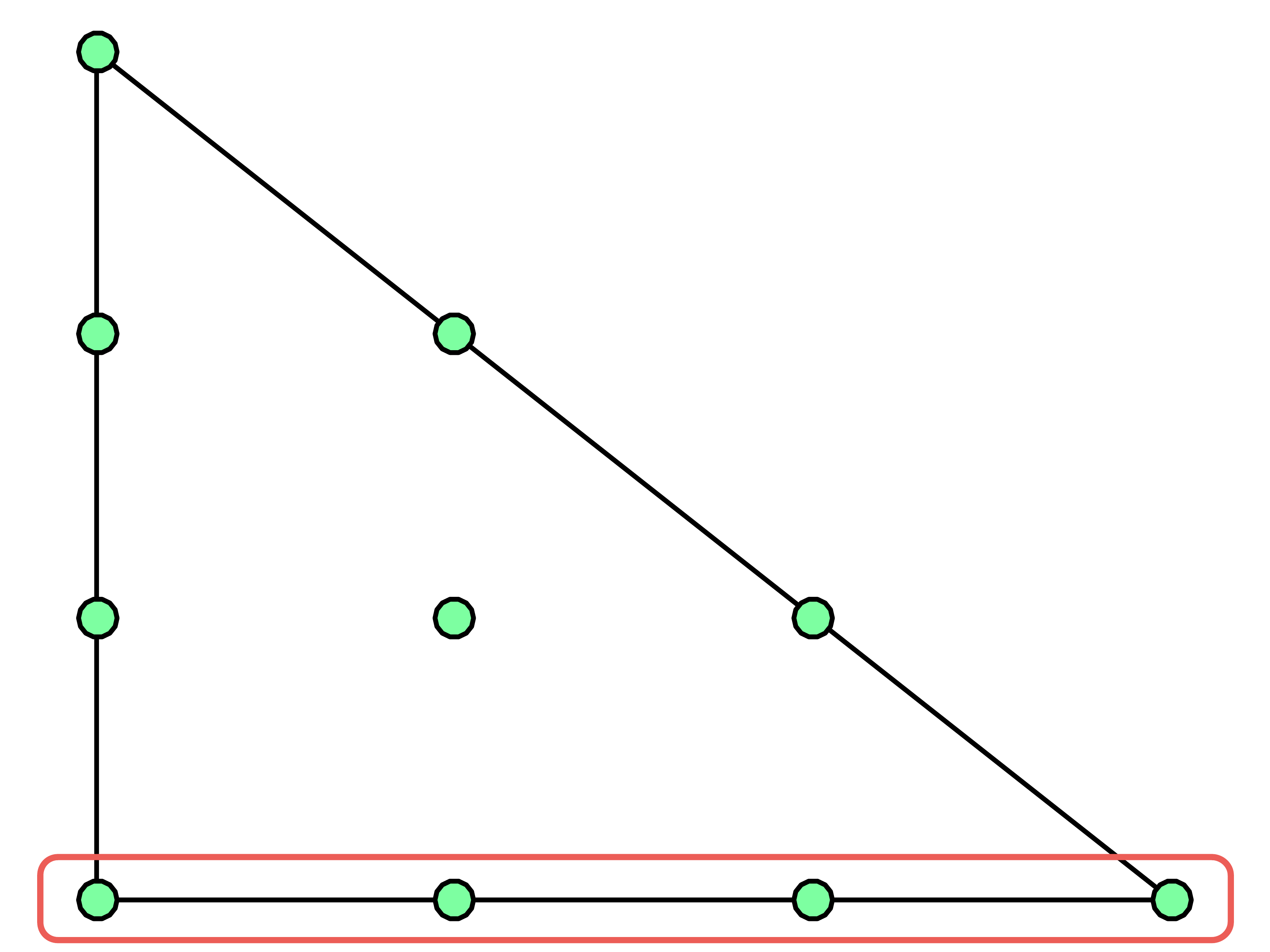}}
\subfloat{\includegraphics[width=.32\textwidth]{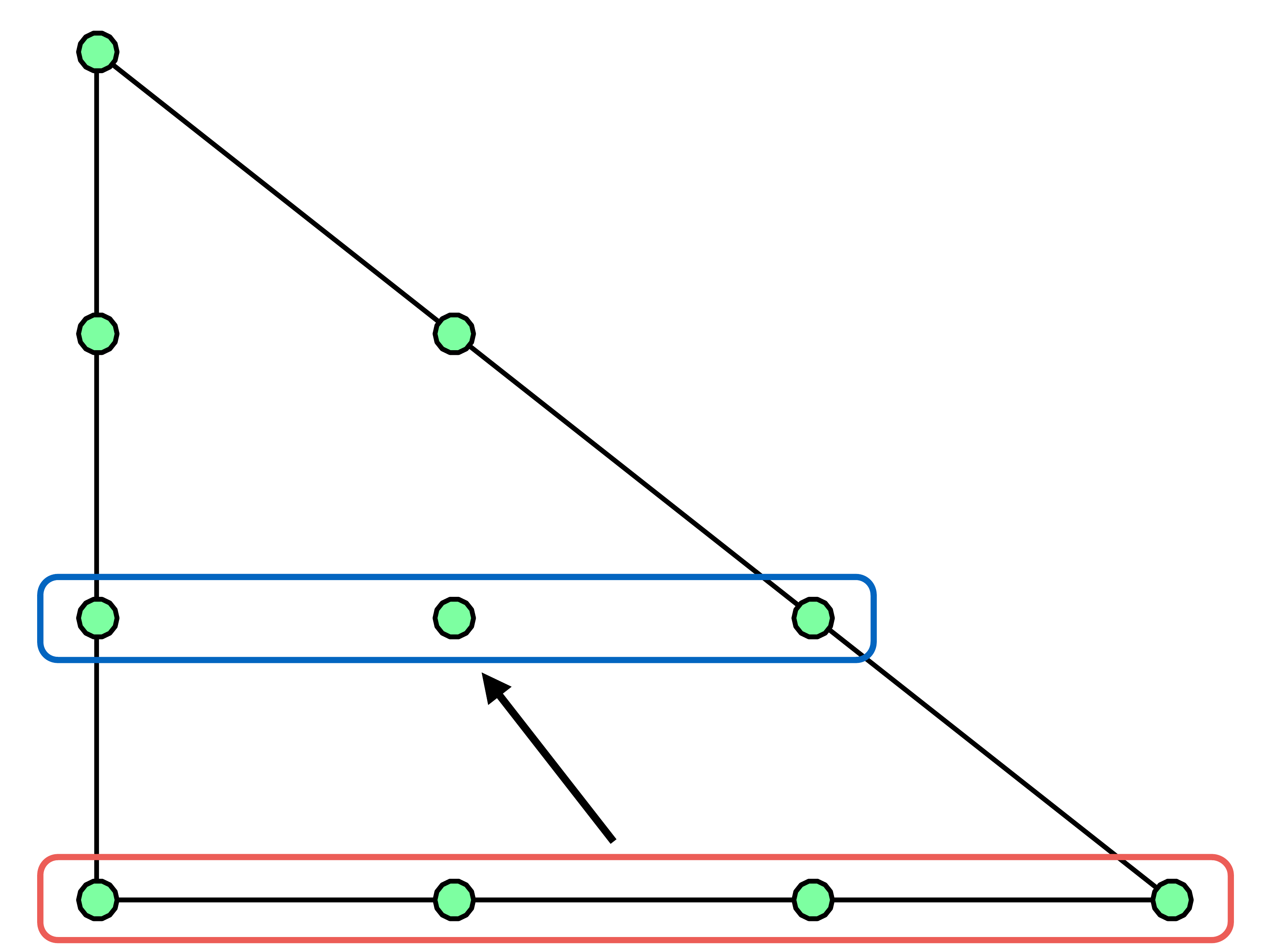}}
\subfloat{\includegraphics[width=.32\textwidth]{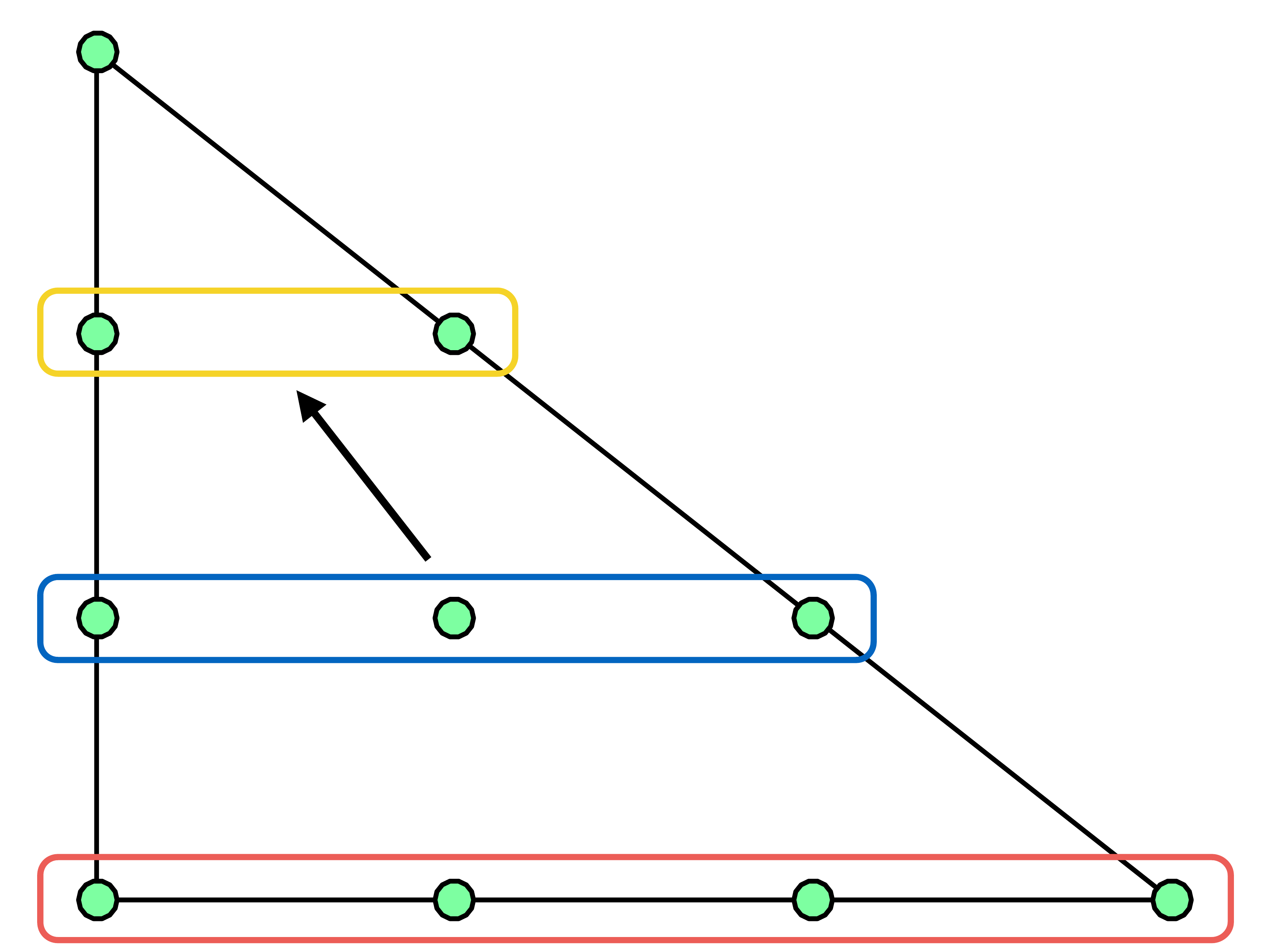}}
\caption{Illustration of the ``slice-by-slice'' optimal-complexity application of the face reduction matrix $\bm{E}_L^f$.}
\label{fig:optlift}
\end{figure}

Despite the optimal asymptotic complexity of Algorithm~\ref{alg:slice_lift}, the implementation of this approach on many-core architectures suffers from several drawbacks.  For example, for implementations where threads are parallelized over degrees of freedom, a synchronization is required after each one-degree reduction to ensure that the result is available in memory before applying the next degree reduction.  This can be costly, as these points of synchronization effectively serialize the code.  We address this cost by introducing a second method of applying the lift matrix which is faster for low orders of approximation.  

We wish to determine $\bm{p}_L$, which results from the application of the lift matrix $\bm{L}$ applied to the flux vector over all tetrahedral faces  
\[
\bm{p}_L = \bm{L}\LRp{\begin{array}{c}
\bm{p}^1 \\ \vdots \\ \bm{p}^4 \end{array}} = \LRp{\bm{L}^1 \middle| \ldots \middle| \bm{L}^4}\LRp{\begin{array}{c}
\bm{p}^1 \\ \vdots \\ \bm{p}^4 \end{array}}.
\]
For convenience, we define the lift reduction matrix $\bm{E}_L$ by concatenating matrices $\bm{E}^f_L$ for each face of the tetrahedron
\[
\bm{E}_L = \LRp{\bm{E}_L^1 \middle| \ldots \middle| \bm{E}_L^4},
\]
where, due to rotational symmetry of the Bernstein-Bezier basis, the block columns $\bm{E}_L^1, \bm{E}_L^2, \bm{E}_L^3, \bm{E}_L^4$ are identical up to row permutation.  Then, $\bm{p}_L$ may be determined in two steps
\begin{align*}
\bm{p}_L^f &= \bm{L}_0 \bm{p}^f, \qquad f = 1,\ldots,4\\
\bm{p}_L &= \bm{E}_L \LRp{\bm{p}_L^1 \middle| \ldots \middle| \bm{p}_L^4}.
\end{align*}
This procedure effectively factorizes the lift matrix $\bm{L}$ into the lift reduction matrix $\bm{E}_L$ and the smaller matrix $\bm{L}_0$.  By exploiting the sparsity of both $\bm{E}_L$ and $\bm{L}_0$, the above factorization can be made more efficient than a direct application of the lift matrix for sufficiently large $N$.  

The sparsity pattern of $\bm{E}_L$ is such that there are at most $N^f_p + 3$ entries per row for a tetrahedron.  In comparison, the face reduction matrix $\bm{E}_L^f$ contains rows with $N^f_p$ non-zero entries.  Thus, despite the fact that the lift reduction matrix $\bm{E}_L$  is the concatenation of four face reduction matrices $\bm{E}_L^f$, the number of non-zeros per row does not increase greatly.  Additionally, the number of non-zeros per row of $\bm{E}_L$ varies less between rows than the number of non-zeros per row of $\bm{E}_L^f$.  This suggests the storage $\bm{E}_L$ in sparse format using two fixed width arrays containing values and column indices, where the width is equal to the maximum number of non-zeros in a row of $\bm{E}_L$.  

Finally, numerical evidence suggests that applying the Bernstein-Bezier lift matrix through the above factorization may be more numerically stable than direct multiplication.  This is discussed in more detail in Section~\ref{sec:roundoff}.

\section{Numerical results}
\label{sec:results}

In the following sections, we examine the behavior of the Bernstein polynomial basis compared to a nodal basis, focusing on numerical stability in Section~\ref{sec:roundoff} and computational efficiency in Sections~\ref{sec:runtime} and \ref{sec:roofline}.

\subsection{Sensitivity to roundoff}
\label{sec:roundoff}

While Bernstein polynomials may offer significant computational advantages compared to nodal polynomials, their numerical stability remains in question.  Since the Bernstein-Bezier basis consists of monomials in barycentric coordinates, the condition number of the mass and Vandermonde matrices increases exponentially with the degree $N$ \cite{marco2007fast, chan2015short}.  Despite this drawback, numerically stable computations are often still possible through algorithms which exploit properties (such as positivity) of the Bernstein-Bezier basis \cite{kiciak2004recursive,marco2007fast}.  The issue of numerical stability is only exacerbated for GPU computations.  While computations on modern GPUs can be done in double precision, peak performance is achievable only through single precision computations.  Since all results in this paper are computed in single precision, we wish to quantify the impact of finite-precision effects on the accuracy of nodal and Bernstein-Bezier bases.  

\begin{figure}
\centering
\subfloat{
\begin{tikzpicture}
\begin{semilogyaxis}[
	legend cell align=left,
	width=.36\textwidth,
	title={Derivative matrices},
    xlabel={Degree $N$},
    ylabel={Condition number},
    xmin=.5, xmax=9.5,
    ymin=.5, ymax=1500,
    xtick={1,2,3,4,5,6,7,8,9},
    legend pos=north west,
    xmajorgrids=true,
    ymajorgrids=true,
    grid style=dashed,
] 
\addplot+[color=blue,mark=*,mark options={fill=markercolor},semithick]
coordinates{(1,2.1166)(2,4.10326)(3,7.8693)(4,15.1262)(5,29.2)(6,56.5943)(7,110.058)(8,214.618)(9,419.473)};
\addplot+[color=red,mark=square*,mark options={fill=markercolor},semithick]
coordinates{(1,2.1166)(2,3.59166)(3,4.58592)(4,5.08895)(5,5.85851)(6,6.49343)(7,7.43361)(8,8.3401)(9,10.2532)};

\legend{Nodal, Bernstein}
\end{semilogyaxis}
\end{tikzpicture}
}
\subfloat{
\begin{tikzpicture}
\begin{semilogyaxis}[
	legend cell align=left,
	width=.36\textwidth,
title={Lift matrix},
    xlabel={Degree $N$},
    xmin=.5, xmax=9.5,
    ymin=.5, ymax=1500,    
    xtick={1,2,3,4,5,6,7,8,9},
    legend pos=north west,
    xmajorgrids=true,
    ymajorgrids=true,
    grid style=dashed,
] 
\addplot+[color=blue,mark=*,semithick,mark options={fill=markercolor}]
coordinates{(1,1)(2,1.29099)(3,1.5)(4,1.67332)(5,1.82574)(6,1.96396)(7,2.09165)(8,2.21108)(9,2.32379)};
\addplot+[color=red,mark=square*,semithick,mark options={fill=markercolor}]
coordinates{(1,1)(2,3.21827)(3,8.30628)(4,14.1564)(5,22.533)(6,36.2175)(7,58.978)(8,100.128)(9,178.895)};

\legend{Nodal, Bernstein}

\end{semilogyaxis}
\end{tikzpicture}
}
\subfloat{
\begin{tikzpicture}
\begin{semilogyaxis}[
	legend cell align=left,
	width=.36\textwidth,
title={$\bm{E}_L$ and $\bm{L}_0$.},
    xlabel={Degree $N$},
    xmin=.5, xmax=9.5,
    ymin=.5, ymax=1500,    
    xtick={1,2,3,4,5,6,7,8,9},
    legend pos=north west,
    xmajorgrids=true,
    ymajorgrids=true,
    grid style=dashed,
] 
\addplot+[color=red,mark=*,semithick,mark options={fill=markercolor}]
coordinates{(1,1.6)(2,2.14286)(3,2.66667)(4,3.18182)(5,3.69231)(6,4.2)(7,4.70588)(8,5.21053)(9,5.71429)};

\addplot+[color=red,mark=square*,semithick,mark options={fill=markercolor}]
coordinates{(1,1.32288)(2,1.91485)(3,2.95099)(4,4.75395)(5,7.90833)(6,13.4748)(7,23.3872)(8,41.1893)(9,73.4078)};

\legend{$\bm{L}_0$ (Bernstein), $\bm{E}_L$ (Bernstein)}

\end{semilogyaxis}
\end{tikzpicture}
}
\caption{Condition numbers (as defined in Equation~\ref{eq:cond})  of derivative and lift matrices $\bm{L}^f$ under both nodal and Bernstein-Bezier bases for $N = 1,\ldots,9$.}
\label{fig:liftRoundoff}
\end{figure}
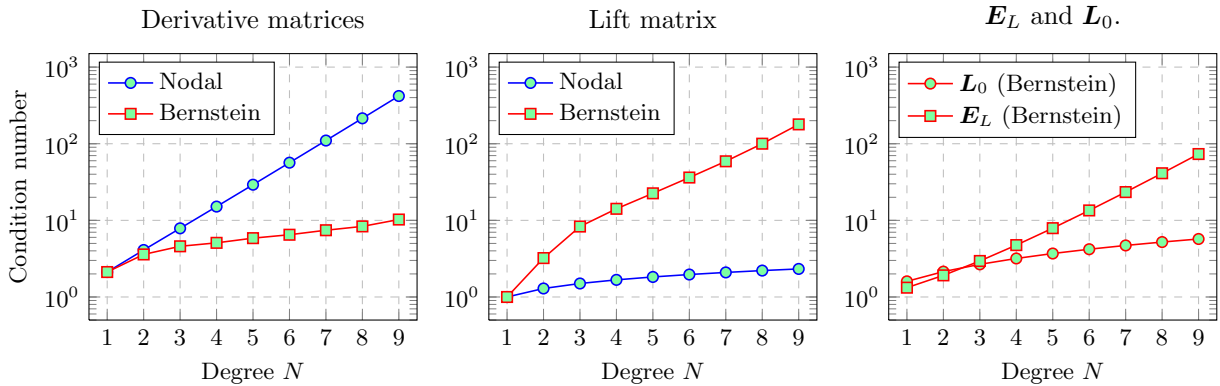

Since we do not invert either matrix, we consider the conditioning of matrix multiplication.  Assuming the $2$-norm, straightforward manipulations \cite{ipsen2009numerical} give
\[
\frac{\nor{\tilde{\bm{A}}{\tilde{\bm{x}}}-\bm{A}\bm{x}}}{\nor{{\bm{Ax}}}} \leq \frac{\nor{\bm{A}}\nor{\bm{x}}}{\nor{\bm{Ax}}} \epsilon_{\bm{Ax}} \leq \frac{\nor{\bm{A}}}{\min_{\bm{x}\neq 0} \frac{\nor{\bm{Ax}}}{\nor{\bm{x}}}} \epsilon_{\bm{Ax}} \leq \frac{\sigma_1}{\sigma_r}\epsilon_{\bm{Ax}}.
\]
where $r$ is the rank of $\bm{A}$, and the relative error $\epsilon_{\bm{Ax}}$ is 
\[
\epsilon = \LRp{\frac{\nor{\tilde{\bm{A}}-\bm{A}}}{\nor{\bm{A}}} + \frac{\nor{\tilde{\bm{x}}-\bm{x}}}{\nor{\bm{x}}} + \frac{\nor{\tilde{\bm{A}}-\bm{A}}}{\nor{\bm{A}}}\frac{\nor{\tilde{\bm{x}}-\bm{x}}}{\nor{\bm{x}}} }.
\]
This implies that the condition number, defined as the ratio of the largest and smallest non-zero singular values
\begin{equation}
\label{eq:cond}
\kappa(A) = \frac{\sigma_1}{\sigma_r}, \qquad \sigma_{r+1} = 0, 
\end{equation}
can be used to measure the conditioning of matrix multiplication.  Figure~\ref{fig:liftRoundoff} compares the conditioning of the derivative and lift matrices using nodal and Bernstein-Bezier bases.  As the degree $N$ increases, the condition number of the Bernstein (barycentric) derivative matrix is observed to remain small relative to the condition number of the nodal derivative matrix.  However, the converse is true for the lift matrix, where the Bernstein lift matrix is observed to be more poorly conditioned than the nodal lift matrix.  
This is somewhat alleviated by decomposing the application of the Bernstein lift into an application of $\bm{L}_0$ and $\bm{E}_L$.  The condition number of $\bm{L}_0$ is very tame, and the condition number of $\bm{E}_L$ is about half an order of magnitude smaller than that of the full Bernstein lift matrix.


An additional source of finite precision effects may stem from large variations in magnitude of positive and negative entries in the lift matrix.  Figure~\ref{fig:liftMinMax} plots the minimum and maximum values of the lift matrix $\bm{L}$ under both Bernstein and nodal bases.  While the minimum and maximum values of nodal lift matrices remain small in magnitude as $N$ increases, the minimum and maximum values of the Bernstein lift matrix increase rapidly in magnitude.  This can exacerbate numerical cancellation when directly computing using the Bernstein lift matrix.   However, the entries of $\bm{L}_0$ and $\bm{E}_L$ do not grow as rapidly with increasing $N$, suggesting that using the decomposition of the lift matrix may be more numerically robust than directly computing with the lift matrix $\bm{L}$.
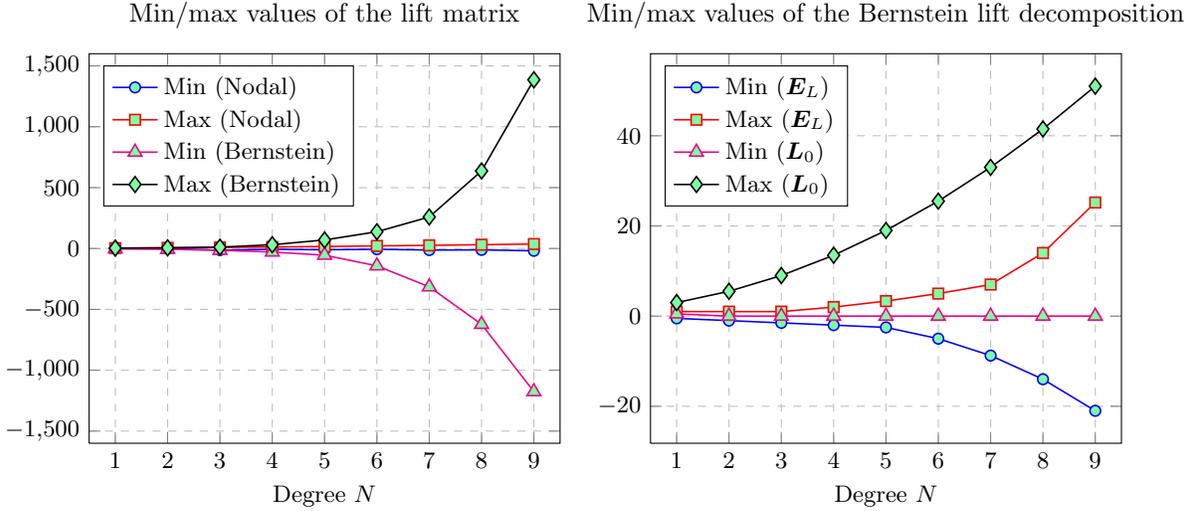
\begin{figure}
\centering
\subfloat{
\begin{tikzpicture}
\begin{axis}[
	legend cell align=left,
	width=.475\textwidth,
	title={Min/max values of the lift matrix},
	xlabel={Degree $N$},
	xmin=.5, xmax=9.5,
	xtick={1,2,3,4,5,6,7,8,9},
	ymin= -1600, ymax=1600,	
	ytick={-1500,-1000,-500,0,500,1000,1500},
	legend pos=north west,
	xmajorgrids=true,
	ymajorgrids=true,
	grid style=dashed,
] 
\addplot+[color=blue,mark=*,semithick,mark options={fill=markercolor}]
coordinates{(1,-2)(2,-1.75)(3,-13.5)(4,-6.52157)(5,-10.3403)(6,-6.08246)(7,-12.846)(8,-11.4069)(9,-18.8242)};
\addplot+[color=red,mark=square*,semithick,mark options={fill=markercolor}]
coordinates{(1,3)(2,5.5)(3,9)(4,12.3666)(5,16.3324)(6,21.1329)(7,25.7652)(8,30.996)(9,36.7226)};
\addplot+[color=magenta,mark=triangle*,semithick,mark options={scale=1.5,fill=markercolor}]
coordinates{(1,-2)(2,-6.5)(3,-15)(4,-29)(5,-55)(6,-142.5)(7,-315)(8,-623)(9,-1176)};
\addplot+[color=black,mark=diamond*,semithick,mark options={scale=1.5,fill=markercolor}]
coordinates{(1,3)(2,5.5)(3,11)(4,31)(5,70)(6,137.5)(7,259)(8,637)(9,1386)};

\legend{Min (Nodal), Max (Nodal), Min (Bernstein), Max (Bernstein)}

\end{axis}
\end{tikzpicture}
}
\subfloat{
\begin{tikzpicture}
\begin{axis}[
	legend cell align=left,
	width=.475\textwidth,
    title={Min/max values of the Bernstein lift decomposition},
    xlabel={Degree $N$},
    xmin=.5, xmax=9.5,
    xtick={1,2,3,4,5,6,7,8,9},
    legend pos=north west,
    xmajorgrids=true,
    ymajorgrids=true,
    grid style=dashed,
] 
\addplot+[color=blue,mark=*,mark options={fill=markercolor},semithick]
coordinates{(1,-0.5)(2,-1)(3,-1.5)(4,-2)(5,-2.5)(6,-5)(7,-8.75)(8,-14)(9,-21)};
\addplot+[color=red,mark=square*,semithick,mark options={fill=markercolor}]
coordinates{(1,1)(2,1)(3,1)(4,2)(5,3.33333)(6,5)(7,7)(8,14)(9,25.2)};
\addplot+[color=magenta,mark=triangle*,semithick,mark options={scale=1.5,fill=markercolor}]
coordinates{(1,0.5)(2,0)(3,0)(4,0)(5,0)(6,0)(7,0)(8,0)(9,0)};
\addplot+[color=black,mark=diamond*,semithick,mark options={scale=1.5,fill=markercolor}]
coordinates{(1,3)(2,5.5)(3,9)(4,13.5)(5,19)(6,25.5)(7,33)(8,41.5)(9,51)};

\legend{Min ($\bm{E}_L$), Max ($\bm{E}_L$), Min ($\bm{L}_0$), Max ($\bm{L}_0$)}

\end{axis}
\end{tikzpicture}
}
\caption{Minimum and maximum values over all entries of the lift matrix using nodal and Bernstein-Bezier bases, as well as minimum and maximum values of entries in the decomposition of the Bernstein lift matrix into $\bm{L}_0$ and $\bm{E}_L$, which grow much less rapidly than the min/max values of the Bernstein lift matrix.}
\label{fig:liftMinMax}
\end{figure}


Finally, we examine numerically the behavior of the error over time for both nodal and Bernstein-Bezier bases.  For a cube domain $\Omega = [-.5,.5]^3$, we compare the $L^2$ error between the approximated and exact solution to the acoustic wave equation with $\rho = \kappa = 1$ 
\[
p(x,y,z,\tau) = \cos(\pi x)\cos(\pi y)\cos(\pi z)\cos(\sqrt{3}\pi \tau)
\]
at various times and orders of approximation $N$.  The initial condition for the Bernstein-Bezier basis is determined by interpolating the solution using the nodal basis, then determining the Bernstein coefficients through a change of basis.  The conditioning of this change of basis depends on the choice of interpolation nodes used, as the transformation between Legendre (modal) and nodal basis is well conditioned for an appropriate choice of points \cite{hesthaven2007nodal}, and the transformation between Legendre and Bernstein polynomials is reasonably conditioned \cite{farouki2000legendre}.  In our experiments, we use the optimized Warp and Blend interpolation points of Warburton \cite{warburton2006explicit}, and observe that the condition number of the Bernstein-Vandermonde matrix to be $O(10^3)$ at $N=9$, implying a loss of at most three digits of accuracy.  If the control points of the Lagrange interpolant are computed in double precision before converting to single precision, the effect of this conditioning is negligible for moderate $N$.  Recent work in \cite{ainsworth2015computing} gives a more algorithms for stable transformations between Lagrange and Bernstein-Bezier bases of arbitrary order.  

Since both nodal and Bernstein-Bezier bases span the same high order polynomial space, they produce the same result in the presence of no roundoff error.  To quantify the effect of roundoff, we examine the behavior of the solution over time for orders of approximation $N = 5,6,7$ using a mesh of 1536 tetrahedral elements.  In all cases, the approximation error is below the single precision threshold in order to ascertain the effect of roundoff error on the kernels for each basis.  Figure~\ref{fig:errOverTime} shows the $L^2$ error over time for each degree $N$.  While the error for both both nodal and Bernstein-Bezier bases remains at approximately single precision from time $\tau \in [0, 25]$, the magnitude of $L^2$ error under the Bernstein-Bezier basis is equal to or less than that of the nodal basis.  This suggests that Bernstein polynomials are not only sufficiently well-conditioned for single precision computations, but that the Bernstein-Bezier basis is at least as numerically stable as nodal polynomials at high orders of approximation.  

\begin{figure}
\centering
\subfloat{
\begin{tikzpicture}
\begin{semilogyaxis}[
	legend cell align=left,
	width=.335\textwidth,
    title={$N = 5$},
    xlabel={Time},
    ylabel={$L^2$ error},
    xmin=-.5, xmax=25.5,    
    xtick={0,5,10,15,20,25},    
    ymin=5e-8,ymax=1e-5,    
    legend pos=south east,
    xmajorgrids=true,
    ymajorgrids=true,
    grid style=dashed,
] 
\addplot+[color=blue,mark=none,mark options={fill=markercolor},semithick]
coordinates{(0.00942223,3.84995e-07)(0.263822,3.77063e-07)(0.518223,4.99659e-07)(0.772623,4.95252e-07)(1.02702,5.11151e-07)(1.28142,6.06889e-07)(1.53582,6.06564e-07)(1.79022,6.29366e-07)(2.04462,7.41239e-07)(2.29902,6.69313e-07)(2.55342,6.96358e-07)(2.80782,7.51253e-07)(3.06222,7.69608e-07)(3.31662,7.20266e-07)(3.57102,7.8324e-07)(3.82543,8.72844e-07)(4.07983,8.01464e-07)(4.33423,1.07615e-06)(4.58863,7.89313e-07)(4.84303,7.77708e-07)(5.09743,8.17493e-07)(5.35183,8.98188e-07)(5.60623,9.01123e-07)(5.86063,8.94138e-07)(6.11503,1.17259e-06)(6.36943,8.44184e-07)(6.62383,1.38205e-06)(6.87823,8.91514e-07)(7.13263,9.73265e-07)(7.38703,9.11068e-07)(7.64143,1.12352e-06)(7.89583,1.14526e-06)(8.15023,1.08759e-06)(8.40463,1.54808e-06)(8.65903,9.1388e-07)(8.91343,1.83281e-06)(9.16783,1.07909e-06)(9.42223,1.81755e-06)(9.67663,1.5794e-06)(9.93103,9.47185e-07)(10.1854,9.57205e-07)(10.4398,1.00801e-06)(10.6942,1.16644e-06)(10.9486,9.65957e-07)(11.203,1.42273e-06)(11.4574,1.04892e-06)(11.7118,1.51281e-06)(11.9662,1.37004e-06)(12.2206,1.38723e-06)(12.475,1.92986e-06)(12.7294,1.00154e-06)(12.9838,2.26187e-06)(13.2382,1.01886e-06)(13.4926,2.41264e-06)(13.747,1.5811e-06)(14.0014,1.22239e-06)(14.2558,1.16729e-06)(14.5102,1.21446e-06)(14.7646,1.52878e-06)(15.019,1.02512e-06)(15.2734,1.8455e-06)(15.5278,1.0822e-06)(15.7822,1.93168e-06)(16.0366,1.52889e-06)(16.291,1.71152e-06)(16.5454,2.09594e-06)(16.7998,1.39752e-06)(17.0542,2.63722e-06)(17.3086,9.91853e-07)(17.563,2.98945e-06)(17.8174,1.40358e-06)(18.0718,2.95084e-06)(18.3262,2.42357e-06)(18.5806,2.39904e-06)(18.835,3.40491e-06)(19.0894,1.39077e-06)(19.3438,3.97567e-06)(19.5982,1.12321e-06)(19.8526,1.14893e-06)(20.107,1.01862e-06)(20.3614,1.27616e-06)(20.6158,1.12612e-06)(20.8702,1.30673e-06)(21.1246,1.4958e-06)(21.379,1.14004e-06)(21.6334,1.84564e-06)(21.8878,1.04765e-06)(22.1422,2.07671e-06)(22.3966,1.42354e-06)(22.651,1.89572e-06)(22.9054,2.06203e-06)(23.1598,1.50363e-06)(23.4142,2.62205e-06)(23.6686,1.05211e-06)(23.923,2.98076e-06)(24.1774,1.32386e-06)(24.4318,2.97754e-06)(24.6862,2.28985e-06)(24.9406,2.46113e-06)};
\addplot+[color=red,mark=none,semithick,mark options={fill=markercolor}]
coordinates{(0.00942223,3.86948e-07)(0.263822,2.35677e-07)(0.518223,3.58057e-07)(0.772623,2.88983e-07)(1.02702,3.13855e-07)(1.28142,3.36094e-07)(1.53582,3.23459e-07)(1.79022,4.07506e-07)(2.04462,3.95739e-07)(2.29902,4.28878e-07)(2.55342,3.92824e-07)(2.80782,4.58387e-07)(3.06222,4.23612e-07)(3.31662,4.54228e-07)(3.57102,4.30268e-07)(3.82543,4.55962e-07)(4.07983,4.91946e-07)(4.33423,5.83775e-07)(4.58863,5.12281e-07)(4.84303,5.91047e-07)(5.09743,4.83349e-07)(5.35183,5.00045e-07)(5.60623,5.26429e-07)(5.86063,5.28736e-07)(6.11503,5.89948e-07)(6.36943,5.22985e-07)(6.62383,6.82256e-07)(6.87823,5.42033e-07)(7.13263,5.82006e-07)(7.38703,5.13774e-07)(7.64143,5.70022e-07)(7.89583,6.08859e-07)(8.15023,6.58064e-07)(8.40463,7.14815e-07)(8.65903,6.94315e-07)(8.91343,9.30685e-07)(9.16783,7.20271e-07)(9.42223,9.46182e-07)(9.67663,9.52157e-07)(9.93103,6.97769e-07)(10.1854,8.1226e-07)(10.4398,5.92856e-07)(10.6942,6.14142e-07)(10.9486,6.27902e-07)(11.203,6.51575e-07)(11.4574,6.05284e-07)(11.7118,6.63185e-07)(11.9662,7.3344e-07)(12.2206,6.77831e-07)(12.475,8.52744e-07)(12.7294,6.23157e-07)(12.9838,1.03062e-06)(13.2382,6.73812e-07)(13.4926,1.20917e-06)(13.747,8.25663e-07)(14.0014,7.43881e-07)(14.2558,7.17358e-07)(14.5102,6.89299e-07)(14.7646,5.77591e-07)(15.019,7.36187e-07)(15.2734,6.76203e-07)(15.5278,7.28733e-07)(15.7822,7.14491e-07)(16.0366,7.85869e-07)(16.291,7.74484e-07)(16.5454,8.90705e-07)(16.7998,8.08551e-07)(17.0542,1.13944e-06)(17.3086,7.77788e-07)(17.563,1.46731e-06)(17.8174,7.79281e-07)(18.0718,1.45417e-06)(18.3262,1.34498e-06)(18.5806,1.23774e-06)(18.835,1.79806e-06)(19.0894,9.05925e-07)(19.3438,2.14399e-06)(19.5982,8.06568e-07)(19.8526,1.45493e-06)(20.107,1.09129e-06)(20.3614,1.02786e-06)(20.6158,1.19968e-06)(20.8702,7.55967e-07)(21.1246,1.02041e-06)(21.379,8.5471e-07)(21.6334,8.7873e-07)(21.8878,7.47256e-07)(22.1422,7.91899e-07)(22.3966,7.49749e-07)(22.651,6.8323e-07)(22.9054,7.77994e-07)(23.1598,7.53276e-07)(23.4142,7.34561e-07)(23.6686,8.6858e-07)(23.923,9.91957e-07)(24.1774,6.46082e-07)(24.4318,1.16941e-06)(24.6862,9.2726e-07)(24.9406,1.01434e-06)};
\legend{Nodal,Bernstein}
\end{semilogyaxis}
\end{tikzpicture}
}
\subfloat{
\begin{tikzpicture}
\begin{semilogyaxis}[
	legend cell align=left,
	width=.335\textwidth,
    title={$N = 6$},
    xlabel={Time},
    xmin=-.5, xmax=25.5,    
    xtick={0,5,10,15,20,25},    
    ymin=5e-8,ymax=1e-5,    
    legend pos=south east,
    xmajorgrids=true,
    ymajorgrids=true,
    grid style=dashed,
] 
\addplot+[color=blue,mark=none,mark options={fill=markercolor},semithick]
coordinates{(0.00942223,1.59337e-07)(0.263822,3.34267e-07)(0.518223,4.27429e-07)(0.772623,5.08211e-07)(1.02702,5.51752e-07)(1.28142,6.07672e-07)(1.53582,6.29049e-07)(1.79022,6.21552e-07)(2.04462,7.72263e-07)(2.29902,6.64154e-07)(2.55342,7.0817e-07)(2.80782,7.17139e-07)(3.06222,7.89352e-07)(3.31662,8.07981e-07)(3.57102,7.35146e-07)(3.82543,1.05491e-06)(4.07983,8.16982e-07)(4.33423,1.2802e-06)(4.58863,8.67775e-07)(4.84303,1.07838e-06)(5.09743,9.90645e-07)(5.35183,9.15434e-07)(5.60623,1.01784e-06)(5.86063,8.36909e-07)(6.11503,1.56378e-06)(6.36943,8.54732e-07)(6.62383,1.3593e-06)(6.87823,1.02163e-06)(7.13263,1.14159e-06)(7.38703,1.21354e-06)(7.64143,1.03392e-06)(7.89583,1.78755e-06)(8.15023,8.94593e-07)(8.40463,1.02153e-06)(8.65903,1.09639e-06)(8.91343,2.20666e-06)(9.16783,1.48997e-06)(9.42223,1.66478e-06)(9.67663,1.80356e-06)(9.93103,1.22332e-06)(10.1854,1.91407e-06)(10.4398,1.00296e-06)(10.6942,1.70362e-06)(10.9486,1.1744e-06)(11.203,1.44629e-06)(11.4574,1.38178e-06)(11.7118,1.10857e-06)(11.9662,1.38231e-06)(12.2206,1.0656e-06)(12.475,2.78469e-06)(12.7294,1.13698e-06)(12.9838,2.5442e-06)(13.2382,1.38374e-06)(13.4926,2.12876e-06)(13.747,1.81043e-06)(14.0014,1.55534e-06)(14.2558,1.97457e-06)(14.5102,1.08551e-06)(14.7646,1.97869e-06)(15.019,1.19172e-06)(15.2734,1.77256e-06)(15.5278,1.3148e-06)(15.7822,1.48718e-06)(16.0366,1.47768e-06)(16.291,1.2237e-06)(16.5454,1.43783e-06)(16.7998,1.03551e-06)(17.0542,1.26394e-06)(17.3086,1.10258e-06)(17.563,1.18557e-06)(17.8174,1.1628e-06)(18.0718,3.33249e-06)(18.3262,3.21885e-06)(18.5806,2.23517e-06)(18.835,3.73378e-06)(19.0894,1.33479e-06)(19.3438,3.7067e-06)(19.5982,1.47732e-06)(19.8526,3.20031e-06)(20.107,2.24583e-06)(20.3614,2.36502e-06)(20.6158,2.86065e-06)(20.8702,1.55678e-06)(21.1246,2.99753e-06)(21.379,1.25036e-06)(21.6334,2.76567e-06)(21.8878,1.70618e-06)(22.1422,2.18238e-06)(22.3966,2.13133e-06)(22.651,1.43188e-06)(22.9054,2.31409e-06)(23.1598,1.22396e-06)(23.4142,2.16151e-06)(23.6686,1.33573e-06)(23.923,1.77086e-06)(24.1774,1.69386e-06)(24.4318,1.32461e-06)(24.6862,1.77528e-06)(24.9406,1.15981e-06)};

\addplot+[color=red,mark=none,semithick,mark options={fill=markercolor}]
coordinates{(0.00942223,9.03235e-08)(0.263822,1.29098e-07)(0.518223,1.93936e-07)(0.772623,2.05495e-07)(1.02702,3.04613e-07)(1.28142,3.41969e-07)(1.53582,3.35877e-07)(1.79022,4.03e-07)(2.04462,4.17363e-07)(2.29902,5.48092e-07)(2.55342,3.64127e-07)(2.80782,5.69421e-07)(3.06222,4.01557e-07)(3.31662,6.27907e-07)(3.57102,7.23612e-07)(3.82543,6.79905e-07)(4.07983,8.85585e-07)(4.33423,5.95572e-07)(4.58863,9.40574e-07)(4.84303,4.33361e-07)(5.09743,8.99871e-07)(5.35183,7.87036e-07)(5.60623,5.99743e-07)(5.86063,1.13867e-06)(6.11503,8.95586e-07)(6.36943,1.18287e-06)(6.62383,5.58382e-07)(6.87823,1.23017e-06)(7.13263,8.70751e-07)(7.38703,1.01032e-06)(7.64143,1.36081e-06)(7.89583,1.10868e-06)(8.15023,1.59034e-06)(8.40463,8.97872e-07)(8.65903,1.63069e-06)(8.91343,8.62928e-07)(9.16783,1.74174e-06)(9.42223,9.13978e-07)(9.67663,1.52592e-06)(9.93103,1.48565e-06)(10.1854,9.99629e-07)(10.4398,1.76496e-06)(10.6942,5.89415e-07)(10.9486,1.81967e-06)(11.203,1.41596e-06)(11.4574,1.38656e-06)(11.7118,1.99872e-06)(11.9662,9.02827e-07)(12.2206,2.36792e-06)(12.475,1.42626e-06)(12.7294,2.01427e-06)(12.9838,6.47492e-07)(13.2382,2.00536e-06)(13.4926,1.29271e-06)(13.747,1.63774e-06)(14.0014,1.80774e-06)(14.2558,9.90484e-07)(14.5102,2.18381e-06)(14.7646,1.00599e-06)(15.019,2.26608e-06)(15.2734,1.61847e-06)(15.5278,1.63951e-06)(15.7822,2.43132e-06)(16.0366,1.06e-06)(16.291,2.76522e-06)(16.5454,1.09777e-06)(16.7998,2.57676e-06)(17.0542,2.01227e-06)(17.3086,2.09511e-06)(17.563,2.92937e-06)(17.8174,9.83384e-07)(18.0718,8.53703e-07)(18.3262,2.17702e-06)(18.5806,1.53061e-06)(18.835,1.56752e-06)(19.0894,2.00298e-06)(19.3438,7.45166e-07)(19.5982,2.20622e-06)(19.8526,1.03844e-06)(20.107,1.8408e-06)(20.3614,1.5947e-06)(20.6158,1.16799e-06)(20.8702,2.11029e-06)(21.1246,7.39822e-07)(21.379,2.16257e-06)(21.6334,1.1482e-06)(21.8878,1.77911e-06)(22.1422,1.99713e-06)(22.3966,1.0908e-06)(22.651,2.40515e-06)(22.9054,7.18939e-07)(23.1598,2.42785e-06)(23.4142,1.5483e-06)(23.6686,1.93166e-06)(23.923,2.35615e-06)(24.1774,1.05027e-06)(24.4318,2.90018e-06)(24.6862,7.87569e-07)(24.9406,2.8777e-06)};

\legend{Nodal,Bernstein}
\end{semilogyaxis}
\end{tikzpicture}
}
\subfloat{
\begin{tikzpicture}
\begin{semilogyaxis}[
	legend cell align=left,
	width=.335\textwidth,
    title={$N = 7$},
    xlabel={Time},
    xmin=-.5, xmax=25.5,    
    xtick={0,5,10,15,20,25},
    ymin=5e-8,ymax=1e-5,
    legend pos=south east,
    xmajorgrids=true,
    ymajorgrids=true,
    grid style=dashed,
] 
\addplot+[color=blue,mark=none,mark options={fill=markercolor},semithick]
coordinates{(0.00942223,1.18969e-07)(0.263822,2.77767e-07)(0.518223,3.59959e-07)(0.772623,4.05725e-07)(1.02702,4.31285e-07)(1.28142,4.75894e-07)(1.53582,4.5325e-07)(1.79022,5.16143e-07)(2.04462,4.92953e-07)(2.29902,5.42498e-07)(2.55342,6.12257e-07)(2.80782,5.77949e-07)(3.06222,5.87402e-07)(3.31662,5.69701e-07)(3.57103,6.21047e-07)(3.82543,6.44097e-07)(4.07983,6.45598e-07)(4.33423,6.42979e-07)(4.58863,6.67263e-07)(4.84303,6.75895e-07)(5.09743,6.42488e-07)(5.35183,8.06589e-07)(5.60623,8.02611e-07)(5.86063,7.19631e-07)(6.11503,7.66199e-07)(6.36943,7.14729e-07)(6.62383,7.30389e-07)(6.87823,7.60222e-07)(7.13263,7.4791e-07)(7.38703,7.28401e-07)(7.64143,8.46617e-07)(7.89583,9.03316e-07)(8.15023,7.64806e-07)(8.40463,8.49642e-07)(8.65903,7.71403e-07)(8.91343,8.32256e-07)(9.16783,8.11927e-07)(9.42223,8.10728e-07)(9.67663,7.85968e-07)(9.93103,7.86787e-07)(10.1854,7.97513e-07)(10.4398,8.28264e-07)(10.6942,9.47881e-07)(10.9486,8.82563e-07)(11.203,1.33296e-06)(11.4574,1.06488e-06)(11.7118,1.07536e-06)(11.9662,1.13701e-06)(12.2206,8.91041e-07)(12.475,1.07126e-06)(12.7294,8.23971e-07)(12.9838,9.47341e-07)(13.2382,9.00958e-07)(13.4926,8.81191e-07)(13.747,9.73315e-07)(14.0014,9.0684e-07)(14.2558,8.86557e-07)(14.5102,9.4084e-07)(14.7646,8.3838e-07)(15.019,9.12886e-07)(15.2734,9.94088e-07)(15.5278,1.05844e-06)(15.7822,1.3637e-06)(16.0366,1.31046e-06)(16.291,1.08454e-06)(16.5454,1.3592e-06)(16.7998,9.54759e-07)(17.0542,1.18105e-06)(17.3086,9.60576e-07)(17.563,1.02685e-06)(17.8174,1.01463e-06)(18.0718,1.02155e-06)(18.3262,1.04755e-06)(18.5806,1.07173e-06)(18.835,9.83862e-07)(19.0894,1.03728e-06)(19.3438,9.79466e-07)(19.5982,9.21262e-07)(19.8526,1.05401e-06)(20.107,9.61514e-07)(20.3614,1.15439e-06)(20.6158,1.13152e-06)(20.8702,1.22349e-06)(21.1246,1.23295e-06)(21.379,1.11801e-06)(21.6334,1.40854e-06)(21.8878,9.43115e-07)(22.1422,2.32792e-06)(22.3966,1.96131e-06)(22.651,1.66291e-06)(22.9054,2.30348e-06)(23.1598,1.18047e-06)(23.4142,2.30781e-06)(23.6686,1.1242e-06)(23.923,2.00211e-06)(24.1774,1.43003e-06)(24.4318,1.5372e-06)(24.6862,1.70213e-06)(24.9406,1.21646e-06)};

\addplot+[color=red,mark=none,semithick,mark options={fill=markercolor}]
coordinates{(0.00942223,2.20193e-07)(0.263822,1.61922e-07)(0.518223,2.10928e-07)(0.772623,1.90155e-07)(1.02702,1.86747e-07)(1.28142,2.17112e-07)(1.53582,1.97358e-07)(1.79022,2.29712e-07)(2.04462,2.1406e-07)(2.29902,2.63323e-07)(2.55342,2.59024e-07)(2.80782,2.83009e-07)(3.06222,2.69977e-07)(3.31662,2.48711e-07)(3.57103,3.5815e-07)(3.82543,2.61698e-07)(4.07983,2.98939e-07)(4.33423,2.79248e-07)(4.58863,3.02818e-07)(4.84303,4.56847e-07)(5.09743,2.99679e-07)(5.35183,3.38105e-07)(5.60623,4.00296e-07)(5.86063,2.99816e-07)(6.11503,2.99638e-07)(6.36943,3.32252e-07)(6.62383,3.51669e-07)(6.87823,3.16237e-07)(7.13263,5.2494e-07)(7.38703,3.39399e-07)(7.64143,3.36897e-07)(7.89583,4.07952e-07)(8.15023,3.5457e-07)(8.40463,3.20798e-07)(8.65903,3.73246e-07)(8.91343,4.33288e-07)(9.16783,3.31251e-07)(9.42223,5.74808e-07)(9.67663,4.11773e-07)(9.93103,6.30739e-07)(10.1854,6.95452e-07)(10.4398,5.3708e-07)(10.6942,1.02049e-06)(10.9486,4.49329e-07)(11.203,6.78877e-07)(11.4574,5.89683e-07)(11.7118,4.23455e-07)(11.9662,5.88911e-07)(12.2206,3.7196e-07)(12.475,4.44646e-07)(12.7294,4.5932e-07)(12.9838,3.97968e-07)(13.2382,4.3413e-07)(13.4926,5.4413e-07)(13.747,4.16291e-07)(14.0014,6.56205e-07)(14.2558,5.44616e-07)(14.5102,6.85559e-07)(14.7646,8.21004e-07)(15.019,5.58738e-07)(15.2734,1.16051e-06)(15.5278,6.5633e-07)(15.7822,5.26931e-07)(16.0366,6.84698e-07)(16.291,4.3554e-07)(16.5454,5.72016e-07)(16.7998,5.35723e-07)(17.0542,4.01191e-07)(17.3086,5.83645e-07)(17.563,5.16393e-07)(17.8174,4.76094e-07)(18.0718,7.11568e-07)(18.3262,4.59762e-07)(18.5806,7.97498e-07)(18.835,6.5153e-07)(19.0894,7.53359e-07)(19.3438,1.00457e-06)(19.5982,4.95595e-07)(19.8526,1.30958e-06)(20.107,5.17539e-07)(20.3614,1.37688e-06)(20.6158,1.03074e-06)(20.8702,1.17885e-06)(21.1246,1.54962e-06)(21.379,7.99616e-07)(21.6334,1.96826e-06)(21.8878,4.47481e-07)(22.1422,1.30195e-06)(22.3966,1.31541e-06)(22.651,7.44483e-07)(22.9054,1.37148e-06)(23.1598,4.62151e-07)(23.4142,1.20753e-06)(23.6686,7.15283e-07)(23.923,8.99061e-07)(24.1774,8.75244e-07)(24.4318,5.76796e-07)(24.6862,8.70818e-07)(24.9406,4.92085e-07)};

\legend{Nodal,Bernstein}
\end{semilogyaxis}
\end{tikzpicture}
}
\caption{$L^2$ errors over time for both nodal and Bernstein-Bezier bases.  Both $N$ and the number of elements are taken to be sufficiently large such that the approximation error is at machine (single) precision.}
\label{fig:errOverTime}
\end{figure}
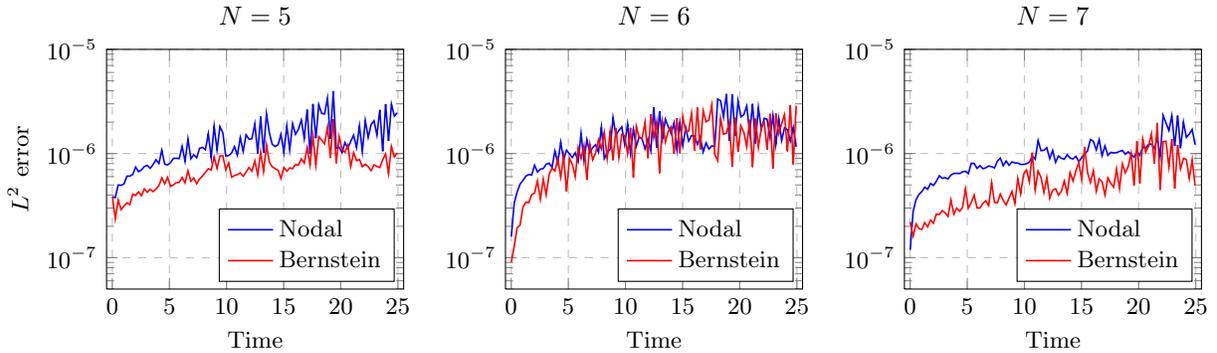


\subsection{Computational implementation}
\label{sec:comp}

The computational work involved in a time-explicit upwind DG solver consists of the evaluation of the right hand side and the solution update.  The implementation of such solvers on GPUs typically divides this work into three relevant kernels
\begin{itemize}
\item A volume kernel, which computes contributions to the right hand side resulting from volumetric derivative terms in the discrete formulation (\ref{eq:discrete_var}).
\item A surface kernel, which computes both numerical fluxes and contributions to the right hand side resulting from trace terms in the discrete formulation (\ref{eq:discrete_var}).  
\item An update kernel, which updates the solution in time. 
\end{itemize}
For the update kernel, we use a low-storage 4th order Runge-Kutta method \cite{carpenter1994fourth} in this work, though any standard explicit time marching scheme may be used.  Since the structure of the update kernel is identical for both nodal and Bernstein polynomials, we do not analyze its performance in detail here.  We note, however, that a different choice of time marching scheme will change the runtime of the update step, and therefore impact any speedups reported in the total runtime.  

In the following sections, we compare the Bernstein volume and surface kernels to reference nodal DG volume and surface kernels.  Common to these kernel are computational parameters $K_V, K_S$ and $K_U$, which refer to the number of elements processed by each thread block.  By tuning these parameters, the number of threads can be made close to a multiple of 32, the number of concurrently processed threads per thread block \cite{klockner2009nodal}.  For all reported results, the block sizes $K_V, K_S$, and $K_U$ have been optimized to minimize runtime.  We also include a comparison between highly optimized nodal DG volume and surface kernels based on blocked matrix multiplication, using strategies adopted in GPU implementations of BLAS routines \cite{golub2012matrix, nath2010improved, nvidia-cublas}.  Due to the partitioning of computational work for each strategy, we will refer to the former implementation as ``node-per-thread'' and the latter as ``element-per-thread''  


\subsubsection{Nodal DG volume and surface kernels}

The implementation of the node-per-thread (NPT) reference volume kernel and surface kernel for nodal DG follows \cite{klockner2009nodal}.  The nodal volume kernel loads three dense differentiation matrices progressively from global device memory to registers via L2/L1 caches. Each thread computes the derivative at a single node, calculating the inner product of a row of each differentiation matrix with the nodal element solution vectors. This row wise reduction typically involves a loop carried dependency of the serial thread reduction, with instruction level parallelism limited to the number of independent fields being differentiated.  The surface kernel retrieves neighboring data on shared faces and computes numerical fluxes.  Similarly to the volume kernel, the surface kernel loads the lift matrix progressively from global memory and applies it to the elemental vector of nodal numerical flux values in an analogous manner to the application of derivative matrices.  

In comparison, the element-per-thread (EPT) kernels use a divide and conquer strategy inspired by GPU implementations of the BLAS matrix-multiplication subroutine \verb+sgemm+ \cite{nvidia-cublas} by effectively extending the one-element-per thread strategy described by Fuhry et\ al.\ in \cite{fuhry2014discontinuous}. In this version of the kernel, the work is partitioned into blockwise matrix-multiplications with the sub-blocks of local matrices staged in shared memory.  Each thread is responsible for applying these local blocks to all (volume or surface) nodes in a single element. To maximize throughput, each thread loads a single nodal value for each field at a time, then increments the derivatives or lifted values at a subset of output nodes in the element.  The results are then staged in an array of thread-local registers.  Block partitioning the matrix-matrix multiplication does require more loads than the NPT algorithm; however, this is offset by the reuse of the manually cached blocks of each matrix for multiple elements.  A significant advantage of this one thread per element approach is that the number of threads can be tuned to exactly match the number of SIMT lanes in the wide vector processing GPU cores. On the Nvidia GTX 980 GPU used for computational tests in this paper, the SIMD width is 32, and since the differentiation matrices are manually cached in shared memory, they are guaranteed to be reused at least 32 times, compared to a much smaller number of times in the NPT nodal volume kernel.  Additionally, by matching the number of collaborative threads to the number of SIMD lanes we do not need to introduce local memory fences when the threads collaboratively load sub-blocks of the differentiation matrices.  Finally, we note that while the behavior of the EPT strategy is similar to that of a DG strategy using CUBLAS \cite{axelGPU2015}, we have observed that CUBLAS is a constant factor faster for all orders tested, though EPT uses less global memory.  

Due to the data layout of the EPT nodal DG kernels, the work of the surface kernel is separated into a ``Lift'' kernel (which applies the lift matrix in a blocked fashion) and an additional ``Flux'' kernel (which computes numerical fluxes and writes them to global memory).  The update kernel is taken to be the same between the NPT nodal, EPT nodal, and Bernstein-Bezier DG kernels.  We note that alternative implementations, such as the fusing of different kernels together, may yield additional speedup.  

\subsubsection{Bernstein-Bezier volume and surface kernels}

Specialized volume and surface kernels must be constructed to efficiently utilize the properties of Bernstein polynomials given in Lemma~\ref{lemma:lemma1} and Section~\ref{sec:lmat}.  While we use a node-per-thread data layout for the Bernstein-Bezier kernels, the application of the derivative and lift matrices is done in a sparse manner.   The implementation of the sparse Bernstein-Bezier volume and lift kernels are described in more detail in Algorithms~\ref{bdgV} and \ref{bdgS}.  The lift matrix may also be applied using the optimal complexity approach outlined in Algorithm~\ref{alg:slice_lift}, which is referred to as the \textit{optimal} Bernstein surface kernel for the remainder of this work.  

 \begin{algorithm}
\caption{Bernstein-Bezier DG volume kernel} 
\label{bdgV} 
\begin{algorithmic}[1]
\ParFor{each element $D^k$} 
\ParFor{volume degrees of freedom $i = 1,\ldots, N_p$} 
\State Load geometric factors and solution variables from global memory into shared memory.
\EndParFor
\EndParFor

\ParFor{each element $D^k$} 
\ParFor{volume degrees of freedom $i = 1,\ldots, N_p$} 
\State Load the four nonzero entries of the $i$th row of the derivative matrices.
\For{$j = 1,\ldots,4$} 
\State Load the indices $c^0_j,\ldots, c^3_j$ such that $(\bm{D}^k)_{i,c^k_j} \neq 0$ for $k = 0,\ldots,d$.
\State Compute expansion coefficients $\LRp{\bm{p}^k}_i$ of solution barycentric derivatives using $\bm{D}^0,\ldots, \bm{D}^3$
\[
\LRp{\bm{p}^k}_i = \LRp{\bm{p}^k}_i + \bm{D}^k_{i,c^k_j} \bm{p}_{c^k_j}. 
\]
\EndFor
\State Compute derivatives with respect to reference coordinates using the chain rule
\[
\pd{p}{r}{} = \frac{1}{2}\LRp{\pd{p}{\lambda_1}{}-\pd{p}{\lambda_0}{}}, \qquad 
\pd{p}{s}{} = \frac{1}{2}\LRp{\pd{p}{\lambda_2}{}-\pd{p}{\lambda_0}{}}, \qquad
\pd{p}{t}{} = \frac{1}{2}\LRp{\pd{p}{\lambda_3}{}-\pd{p}{\lambda_0}{}}.
\]
\State Assemble right hand side contributions and write to global memory.
\EndParFor
\EndParFor
\end{algorithmic}
\end{algorithm}
\begin{algorithm}
\caption{Bernstein-Bezier DG surface kernel (non-optimal)} 
\label{bdgS} 
\begin{algorithmic}[1]
\ParFor{each element $D^k$} 
\For{faces $f = 1,\ldots,4$}
\ParFor{trace degrees of freedom $i = 1,\ldots, N^f_p$} 
\State Load surface normals, Jacobian scalings, and solution traces from global memory.
\State Compute coefficients $\bm{p}^f_{i}$ of the numerical flux for face $f$. 
\State Store geometric factors/fluxes in shared memory.
\EndParFor
\EndFor
\EndParFor

\ParFor{each element $D^k$} 
\For{faces $f = 1,\ldots,4$}
\ParFor{trace degrees of freedom $i = 1,\ldots, N^f_p$} 
\For{$j = 1,\ldots, 7$}
\State Load values and column indices $c_j$ where $(\bm{L}_0)_{i,c_j}\neq 0$, apply $\bm{L}_0$ to $\bm{p}^f$
\[
\LRp{\bm{p}_L^f}_{i} = \LRp{\bm{p}_L^f}_{i} + \LRp{\bm{L}_0}_{i,c_j} \bm{p}^f_{c_j}
\]
\EndFor
\State Store $\LRp{\bm{p}^f_L}_i$ in shared memory.  
\EndParFor
\EndFor
\EndParFor

\ParFor{each element $D^k$} 
\ParFor{volume degrees of freedom $i = 1,\ldots, N_p$} 
\For{faces $f = 1,\ldots,4$}
\For{$j = 1,\ldots,N^f_p + 3$}
\State Load values and column indices $c_j$ where $\LRp{\bm{E}^f_L}_{i,c_j}\neq 0$, apply $\bm{E}^f_L$ to $\bm{p}_L^f$
\[
\bm{R}_i = \bm{R}_i + \LRp{\bm{E}^f_L}_{i,c_j} \LRp{\bm{p}_L^f}_{c_j}
\]
\EndFor
\EndFor
\State Accumulate right hand side contributions $\bm{R}_i$ in global memory.
\EndParFor
\EndParFor

\end{algorithmic}
\end{algorithm}

We experimented with several different ways to apply the Bernstein derivative and lift matrices.  The most straightforward approach is to store column indices for each row, which may then be loaded in a data-parallel manner over threads.  Each thread computes the dot product of a sparse row of a derivative matrix with the local solution vector.  For each row, four column indices for each derivative matrix (with a total of four derivatives with respect to barycentric coordinates) are loaded, as well as a single set of four floating point numbers containing the values of the derivative matrices (which are identical across all derivatives).  The implementation of the surface kernel also followed this pattern.  This compressed storage significantly improves performance by decreasing the required data movement to apply local matrices.  Additionally, indices were packed into 128-bit \verb+int4+ and \verb+float4+ arrays to minimize the number of required memory transactions.  

For the volume kernel, since the sparsity pattern of the derivative operators is explicitly known, another option is to simply store the barycentric tuple $(i,j,k,l)$ for each row of the matrix and determine the four non-zero column indices based on known formulas.  This may be additionally compressed by noting that $i,j,k,l \leq N$.  For most reasonable values of $N$, these four indices may be bitmasked and stored within the bits of a single integer.  This compresses the representation of the derivative matrix to the storage of a single integer array of size $N_p$, which is then unpacked and used to determine non-zero column indices and values of derivative matrices.  However, we found in practice that this resulted in a slower application of the derivative matrix than simply loading four column indices per row of $\bm{D}^i$, possibly due to extra integer arithmetic in determining non-zero column indices.  

Finally, while a node-per-thread strategy has been adopted for the Bernstein-Bezier basis, it is also possible to implement the application of derivative and lift matrices using an element-per-thread strategy.  However, due to the sparsity of the Bernstein derivative and lift matrices, it is unclear how to implement the block-partitioning used in the element-per-thread strategy.  We therefore avoided the use of block-partitioning when constructing element-per-thread versions of Bernstein-Bezier kernels.  Since the element-per-thread version of the Bernstein-Bezier volume kernel is more straightforward to extrapolate, we discuss an element-per-thread implementation of the Bernstein-Bezier lift kernel, which applies the lift reduction matrix $\bm{E}_L$.  

We experimented with the application of the lift reduction matrix $\bm{E}_L$ as both a sparse matrix and as the product of degree reduction operators, as discussed in Section~\ref{sec:lmat}.  Under an element-per-thread parallelization, the latter strategy becomes more computationally attractive, since consecutive one-degree reductions can be done on each thread without requiring synchronizations.  By using a size $N^f_p$ thread-local register array and applying the lift reduction matrix face-by-face, it possible to perform in-place applications of degree reductions using three nested loops.  Additionally, recalling that the face reduction matrices $\bm{E}^1_f,\ldots,\bm{E}^3_L$ are row permutations of $\bm{E}^0_L$, these permutations can be explicitly determined by looping over appropriate barycentric indices when writing out to global memory.  Unfortunately, despite these optimizations, both element-per-thread implementations of the lift reduction matrix proved much slower than the blocked element-per-thread nodal DG lift kernel at higher $N$ and did not yield significant speedup at lower $N$, possibly due to the heavy use of register memory.  We will explore more efficient element-per-thread implementations of Bernstein-Bezier kernels in future work.  

\subsection{Runtime comparisons}
\label{sec:runtime}

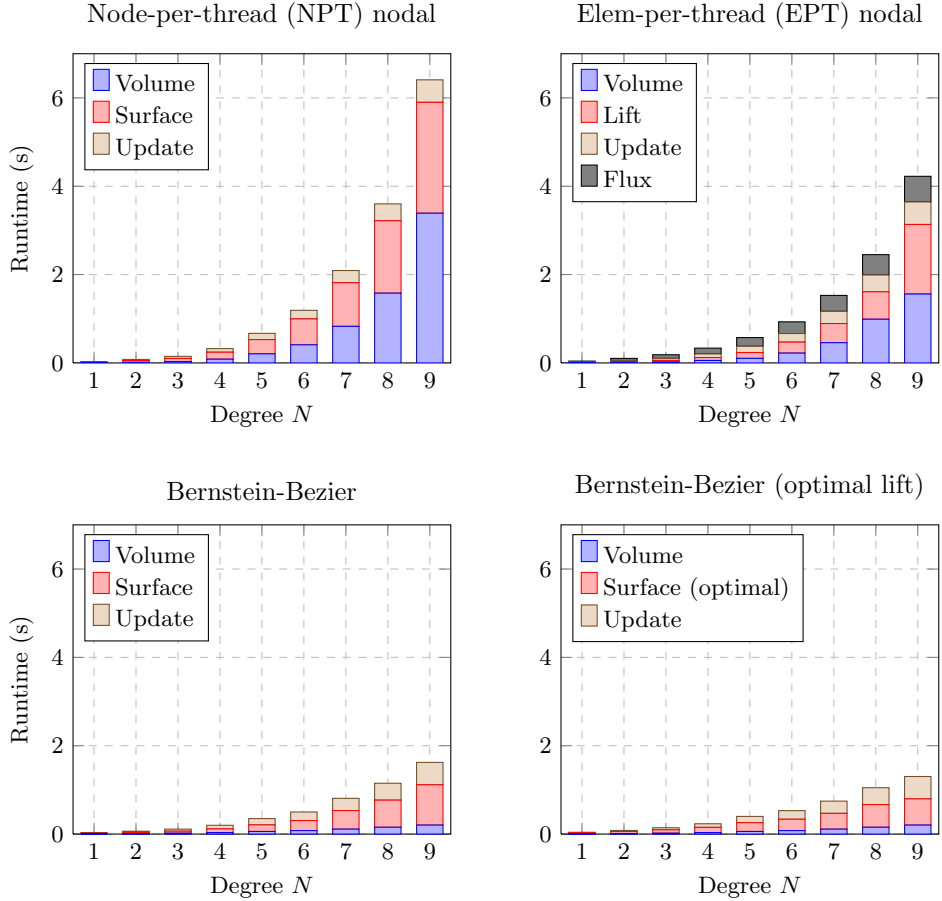
\begin{figure}
\centering
\subfloat{
\begin{tikzpicture}
\begin{axis}[
	width=.4\textwidth,
	legend cell align=left,
	title={Node-per-thread (NPT) nodal},
	xlabel={Degree $N$},
	ylabel={Runtime (s)},
	xmin=.5, xmax=9.5,
	ymin=0,ymax=7,
	ybar stacked,
	xtick={1,2,3,4,5,6,7,8,9},
	legend pos=north west,
	xmajorgrids=true,
	ymajorgrids=true,
	grid style=dashed,
] 
\addplot coordinates{(1,0.00529)(2,0.0135)(3,0.0302)(4,0.0854)(5,0.206)(6,0.41)(7,0.829)(8,1.58)(9,3.39)};
\addplot coordinates{(1,0.0121)(2,0.0403)(3,0.0722)(4,0.158)(5,0.322)(6,0.587)(7,0.987)(8,1.64)(9,2.51)};
\addplot coordinates{(1,0.009394)(2,0.025607)(3,0.046205)(4,0.080841)(5,0.142045)(6,0.19389)(7,0.27628)(8,0.381)(9,0.5088)};

\legend{Volume, Surface, Update}

\end{axis}
\end{tikzpicture}
}
\hspace{2em}
\subfloat{
\begin{tikzpicture}
\begin{axis}[
	width=.4\textwidth,
	legend cell align=left,
	title={Elem-per-thread (EPT) nodal},
	xlabel={Degree $N$},
	xmin=.5, xmax=9.5,
	ymin=0,ymax=7,
	ybar stacked,
	xtick={1,2,3,4,5,6,7,8,9},
	legend pos=north west,
	xmajorgrids=true,
	ymajorgrids=true,
	grid style=dashed,
] 
\addplot 
coordinates{(1,0.005)(2,0.0107)(3,0.0205)(4,0.056)(5,0.104)(6,0.223)(7,0.4585)(8,0.99)(9,1.56)};
\addplot 
coordinates{(1,0.00935)(2,0.02035)(3,0.03685)(4,0.0665)(5,0.1315)(6,0.2485)(7,0.4335)(8,0.62)(9,1.575)};
\addplot 
coordinates{(1,0.009394)(2,0.025607)(3,0.046205)(4,0.080841)(5,0.142045)(6,0.19389)(7,0.27628)(8,0.381)(9,0.5088)};
\addplot 
coordinates{(1,0.0182)(2,0.0461)(3,0.08203)(4,0.13225)(5,0.1964)(6,0.2651)(7,0.3603)(8,0.461)(9,0.5814)};
\legend{Volume, Lift, Update, Flux}

\end{axis}
\end{tikzpicture}
}
\\
\subfloat{
\begin{tikzpicture}
\begin{axis}[
	width=.4\textwidth,
	legend cell align=left,
	title={Bernstein-Bezier},
	xlabel={Degree $N$},
	ylabel={Runtime (s)},
	xmin=.5, xmax=9.5,
	ymin=0,ymax=7,
	ybar stacked,
	xtick={1,2,3,4,5,6,7,8,9},
	legend pos=north west,
	xmajorgrids=true,
	ymajorgrids=true,
	grid style=dashed,
] 
\addplot 
coordinates{(1,0.00564)(2,0.0119)(3,0.0203)(4,0.034)(5,0.0593)(6,0.0791)(7,0.112)(8,0.155)(9,0.204)};
\addplot 
coordinates{(1,0.0148)(2,0.0276)(3,0.0459)(4,0.0842)(5,0.149)(6,0.227)(7,0.42)(8,0.614)(9,0.912)};
\addplot 
coordinates{(1,0.0093767)(2,0.0255921)(3,0.04623)(4,0.081)(5,0.14229)(6,0.194133)(7,0.277041)(8,0.38163)(9,0.50694)};

\legend{Volume, Surface, Update}

\end{axis}
\end{tikzpicture}
}
\hspace{2em}
\subfloat{
\begin{tikzpicture}
\begin{axis}[
	width=.4\textwidth,
	legend cell align=left,
	title={Bernstein-Bezier (optimal lift)},
	xlabel={Degree $N$},
	xmin=.5, xmax=9.5,
	ymin=0,ymax=7,
	ybar stacked,
	xtick={1,2,3,4,5,6,7,8,9},
	legend pos=north west,
	xmajorgrids=true,
	ymajorgrids=true,
	grid style=dashed,
] 
\addplot 
coordinates{(1,0.00564)(2,0.0119)(3,0.0203)(4,0.034)(5,0.0593)(6,0.0791)(7,0.112)(8,0.155)(9,0.204)};
\addplot 
coordinates{(1,0.026)(2,0.0455)(3,0.0771)(4,0.1182)(5,0.1983)(6,0.2573)(7,0.3572)(8,0.5135)(9,0.5926)};
\addplot 
coordinates{(1,0.0093767)(2,0.0255921)(3,0.04623)(4,0.081)(5,0.14229)(6,0.194133)(7,0.277041)(8,0.38163)(9,0.50694)};

\legend{Volume, Surface (optimal), Update}

\end{axis}
\end{tikzpicture}
}

\caption{Per-kernel runtimes for nodal and Bernstein-Bezier bases.  Runtimes are recorded for ten RK4 timesteps on a mesh of 98304 elements.}
\label{fig:stacks}
\end{figure}

Due to optimizations which take advantage of sparsity, the Bernstein-Bezier volume and surface kernels introduce some overhead compared to the volume and surface kernels for a nodal basis.  As a result, the runtime of the Bernstein volume and surface kernels is slower at $N=1$.  However, as noted previously, the Bernstein-Bezier basis may also utilize nodal volume and surface kernels due to each Bernstein-Bezier basis function being associated with an equispaced node on the tetrahedron.  Thus, Bernstein polynomials can always be made to perform \textit{at least} as well as nodal polynomials for any order of approximation.  

We run timing tests by performing fifty right hand side evaluations (ten timesteps using $4$th order RK) on a mesh of size $K = 98304$ for orders of approximation $N = 1,\ldots,9$.  Computational parameters are optimized for each different value of $N$.  Figure~\ref{fig:stacks} shows a kernel-by-kernel breakdown of the total runtime, and Figure~\ref{fig:speedup} shows the achieved speedup (relative to NPT and EPT nodal kernels) attained by using a Bernstein-Bezier basis.  We note that we have only covered a small number of implementations of the reference nodal kernels.  Other implementations can also be adopted to improve performance at higher orders.  For example, the use of an optimized matrix-multiplication library is observed to achieve a total speedup of $1.6\times$ for $N=6$ to $N=8$ \cite{axelGPU2015}.  

\pgfplotstableread[col sep=space]{
N V S T L TL
1   9.3794e-01   8.1757e-01   8.9829e-01     0.4654    0.6530    
2   1.1345e+00   1.4601e+00   1.2199e+00  0.8857    0.9568    
3   1.4877e+00   1.5730e+00   1.3218e+00  0.9364    1.0346    
4   2.5118e+00   1.8765e+00   1.6277e+00  1.3367    1.3904    
5   3.4739e+00   2.1611e+00   1.9112e+00  1.6238    1.6756    
6   5.1833e+00   2.5859e+00   2.3807e+00  2.2814    2.2447    
7   7.4018e+00   2.3500e+00   2.5861e+00  2.7632    2.8038    
8   1.0194e+01   2.6710e+00   3.1296e+00  3.1938    3.4291    
9   1.6618e+01   2.7522e+00   3.9489e+00  4.2356    4.9165
      }\runtimeNaive
   
\pgfplotstableread[col sep=space]{
N V S T L TL
1    0.8865    1.8615    0.8886  1.0596    0.6459    
2    0.8992    2.4076    1.1769  1.4604    0.9231    
3    1.0099    2.5900    1.2355  1.5419    0.9671    
4    1.6471    2.3605    1.4801  1.6815    1.2643    
5    1.7538    2.2007    1.6203  1.6536    1.4205    
6    2.8192    2.2626    2.0068  1.9961    1.8922    
7    4.0938    1.8900    2.1282  2.2223    2.3073    
8    6.3871    1.7606    2.6168  2.1052    2.8673    
9    7.6471    2.3645    2.8213  3.6389    3.5126
}\runtimeBlocked
   
\begin{figure}
\centering
\subfloat{
\begin{tikzpicture}
\begin{axis}[
	width=.475\textwidth,
	legend cell align=left,
	title={Speedup of Bernstein over NPT nodal DG},
	xlabel={Degree $N$},
	ylabel={Time (NPT)/Time (Bernstein)},
	xmin=.5, xmax=9.5,
	ymin=0,ymax=17,	
        ybar=2*\pgflinewidth,
    bar width=2.5pt,
	xtick={1,2,3,4,5,6,7,8,9},
	ymin=0,
	legend pos=north west,
	ymajorgrids=true,
	grid style=dashed,
] 
\addplot table[x=N, y=V] from \runtimeNaive;
\addplot table[x=N, y=S] from \runtimeNaive;
\addplot table[x=N, y=T] from \runtimeNaive;
\addplot table[x=N, y=L] from \runtimeNaive;
\addplot table[x=N, y=TL] from \runtimeNaive;
\addplot+[draw=black,line legend, very thick,smooth,dashed] coordinates{(0,1)(10,1)};

\legend{Volume, Surface, Total,  Surface (optimal), Total (optimal), Reference (no speedup)}
\end{axis}
\end{tikzpicture}
}
\subfloat{
\begin{tikzpicture}
\begin{axis}[
	width=.475\textwidth,
	legend cell align=left,
	xlabel={Degree $N$},
	ylabel={Time (EPT)/Time (Bernstein)},
	title={Speedup of Bernstein over EPT nodal DG},
	xmin=.5, xmax=9.5,
	ymin=0,ymax=17,
        ybar=2*\pgflinewidth,
    bar width=2.5pt,
	xtick={1,2,3,4,5,6,7,8,9},
	ymin=0,
	legend pos=north west,
	ymajorgrids=true,
	grid style=dashed,
] 
\addplot table[x=N, y=V] from \runtimeBlocked;
\addplot table[x=N, y=S] from \runtimeBlocked;
\addplot table[x=N, y=T] from \runtimeBlocked;
\addplot table[x=N, y=L] from \runtimeBlocked;
\addplot table[x=N, y=TL] from \runtimeBlocked;
\addplot+[draw=black,line legend, very thick,smooth,dashed] coordinates{(0,1)(10,1)};
\legend{Volume, Surface (Lift + Flux), Total, Surface (optimal), Total (optimal), Reference (no speedup)}
\end{axis}
\end{tikzpicture}
}
\caption{Ratio of runtimes of volume/surface kernels and total right hand side evaluation achieved by using a Bernstein-Bezier basis instead of nodal polynomials.}
\label{fig:speedup}
\end{figure}
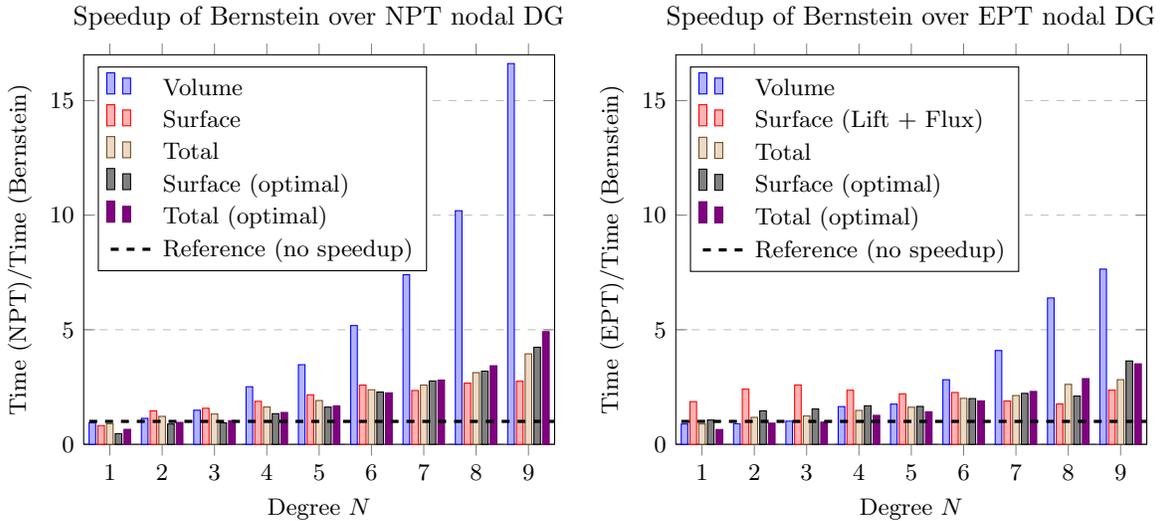


The throughput performance of the Bernstein-Bezier volume kernel compares very favorably with both the NPT and EPT nodal volume kernels.  In comparison to the NPT nodal kernels, the Bernstein volume kernel achieves $2.5\times$ speedup at $N=4$ and increases to over $16\times$ speedup at $N=9$.  In comparison to the EPT nodal kernels, the Bernstein volume kernel achieves a $1.6\times$ speedup at $N=4$, which increases to a $7.6\times$ speedup at $N=9$.  At medium to high order, the sparsity of the Bernstein-Bezier differentiation matrices frees up considerable L1/L2 cache space.  This advantageous because, while the L1/L2 caching of derivative matrices in the NPT nodal volume kernel may improve performance, GPU cache latency far exceeds that of shared and register memory \cite{fatahalian2004understanding, volkov2008benchmarking}.  Additionally, the multiplication of each row entry with the corresponding degrees of freedom can be done with a fixed number of floating point operations using a limited number of registers, and the accumulation of these multiplications can be scheduled to take advantage of instruction level parallelism.  Finally, the reduced storage costs for sparse derivative matrices also increases occupancy by decreasing register pressure compared to the NPT and EPT nodal volume kernels.  

The surface kernel speedups are more modest.  At degrees $N = 4$ and $N=5$, the Bernstein surface kernel is roughly two times faster than the NPT nodal surface kernel.  This improves as $N$ increases, reaching a $2.75 \times$ speedup at $N=9$.  In comparison to the blocked EPT nodal surface kernel (counting both kernel runtimes for the computational of the numerical flux and application of the lift operator), the performance of the Bernstein surface kernel is roughly $2\times$ faster from $N=2$ to $N=6$.  For $N>6$, the Bernstein surface kernel is only around $1.6\times$ faster than the EPT nodal surface kernel.  This may be due in part to the separation of the flux computation from the lift application, which introduces additional global memory transactions.  

Asymptotically, the Bernstein surface kernel is only 4 times faster than the nodal surface kernel, since the cost of applying $\bm{E}_L$ involves multiplication by a sparse matrix with approximately $N^f_p$ entries per row.  In comparison, the lift matrix $\bm{L}$ contains $4N^f_p$ entries per row.  This asymptotic speedup is improved by utilizing the optimal Bernstein surface kernel, which results in an $O(N^3)$ cost in three dimensions.  Since the complexity of the nodal surface kernel is $O(N^{5})$ in three dimensions, the speedup of the optimal Bernstein surface kernel over the nodal surface kernel does not asymptote as $N$ increases.

The total runtime of all the Bernstein kernels (volume, surface, update) involved in one right hand side evaluation (volume, surface, and update) is improved at higher orders as well.  At $N=5$, one Bernstein right hand side is roughly twice as fast as a nodal right hand side evaluation.  At $N = 8$, the Bernstein right hand side evaluation is $3.13\times$ faster than the nodal right hand side evaluation, while at $N=9$, the Bernstein evaluation is roughly $4\times$ faster than the nodal evaluation.  If the optimal Bernstein surface kernel is used, the total Bernstein right hand side evaluation becomes $5\times$ faster than the total nodal right hand side evaluation.  

If Bernstein total runtime is compared to the total runtime of the EPT nodal kernels, the speedup is more modest, achieving $2\times$ speedup $N=6$ and $2.8\times$ speedup at $N=9$.  The speedup of the total Bernstein runtimes over the total EPT nodal runtime improves to $3.5\times$ at $N=9$ if the optimal Bernstein surface kernel is used.


\subsection{Performance and roofline analysis}
\label{sec:roofline}

In this section, we present results which quantify the difference in computational performance between implementations of nodal and Bernstein polynomials.  All results were run on an Nvidia GTX 980, and the solvers was implemented in the Open Concurrent Compute Abstraction framework (OCCA) \cite{medina2014occa,medina2015okl} for clarity and portability.  

Figures~\ref{fig:gflops} and \ref{fig:bw} show the profiled computational performance and bandwidth of the nodal NPT, nodal EPT, and Bernstein kernels.  The two implementations of nodal DG kernels exhibit very different performance characteristics.  NPT nodal volume and surface kernels achieve performance of roughly a teraflop for $N>4$, while the bandwidth of both kernels decreases further and further as $N$ increases.  In comparison, the EPT nodal kernels achieve around twice the computational performance of the NPT nodal kernels (achieving and maintaining close to two teraflops for $N>5$) by trading higher effective bandwidth usage for improved throughput.  The bandwidth of the EPT nodal kernels is also improved at higher orders of approximation, though this is also in part due to the redundant loading of nodal solution values.  However, the performance of both kernels are limited by their extensive use of shared memory or memory caches for staging matrix blocks or solution nodal values.  

In comparison to the NPT nodal kernels, the Bernstein kernels show lower computational performance and improved throughput at high orders of approximation.  The volume kernel achieves a lower performance of around $600$ GFLOPS/s at all orders, but the bandwidth does not degrade as $N$ increases, remaining steady near the peak bandwidth.  The surface kernel behaves similarly, showing decreased computational performance with increased throughput, though the bandwidth still decreases at sufficiently high $N$.   In comparison, the optimal Bernstein surface kernel sustains near-constant GFLOPS/s and bandwidth as $N$ increases, though at a much lower percentage of the peak performance.  

\pgfplotstableread[col sep=space]{
N V S F
1    0.3067    0.2019    0.0584
 2   0.6890    0.4637    0.0461
3    1.2947    0.8537    0.0432
4    1.3824    1.2417    0.0402
5    1.8976    1.4066    0.0379
6    1.9774    1.4887    0.0374
7    1.9683    1.5675    0.0354
8    1.7695    1.8836    0.0346
9    1.9443    1.2083    0.0335
}\GFLOPSBlockNodal

\pgfplotstableread[col sep=space]{
N V S F
1 167.0430  172.5100  133.3820
 2 166.2430  172.1310  142.1890
3  164.0040  172.0910  138.8410
4  111.7100  160.8480  127.9770
5  131.0690  169.5880  123.7740
6  121.5190  131.0820  127.9820
7  124.2650  134.2170  118.5400
8  157.7240  157.9600  115.3520
9  122.9810  122.0410  115.9380
}\BWBlockNodal

\pgfplotstableread[col sep=space]{
N V S
1  0.3048    0.2552
 2   0.5607    0.3123
  3  0.8919    0.5365
 4   0.9145    0.6203
 5   0.9433    0.6689
 6   1.0493    0.7259
 7   1.0487    0.7876
 8   1.0338    0.8112
 9   0.8529    0.8612
 }\GFLOPSNodal

\pgfplotstableread[col sep=space]{
N V Vopt S
1  1.6440e+02   1.6632e+02   1.3286e+02
2   1.5096e+02   1.6659e+02   9.4100e+01
3   1.0639e+02   1.6105e+02   8.0230e+01
4   6.8770e+01   9.6448e+01   6.6073e+01
 5  4.7656e+01   1.2547e+02   5.3975e+01
6   3.3726e+01   1.1956e+02   3.9934e+01
7   2.3302e+01   1.1047e+02   3.2720e+01
8   1.7389e+01   1.0452e+02   2.6895e+01
9   1.0300e+01   1.4811e+02   2.2290e+01
   }\BWNodal
   
\pgfplotstableread[col sep=space]{
N V S L
1    0.4218    0.1847   0.0961    
2    0.4998    0.3181   0.1574    
3    0.5860    0.4434   0.1810    
4    0.6122    0.4898   0.1930    
5    0.5616    0.5256   0.1742    
6    0.6316    0.6190   0.1926    
7    0.6372    0.5710   0.1909    
8    0.6331    0.6379   0.1769    
9    0.6414    0.6747   0.1989
}\GFLOPSBern

\pgfplotstableread[col sep=space]{
N V S L
1  156.0280  108.2120   12.9060
2  167.7100  132.0770   78.4900
3  167.4290  133.0320   79.5690
4  168.5340  118.4780   93.4570
5  168.7020  118.7510   83.6790
6  170.3240  104.9450   95.8590
7  170.4290   76.6690   89.7100
8  169.2200   70.1780   94.3310
9  170.6090   61.6120   97.9680
}\BWBern

\begin{figure}
\centering
\subfloat{
\begin{tikzpicture}
\begin{axis}[
	width=.36\textwidth,
	legend cell align=left,
	legend style={font=\tiny},
	legend image post style={scale=0.5},
	title={Performance (NPT nodal)},
	xlabel={Degree $N$},
	ylabel={TFLOPS/s},
	xmin=.5, xmax=9.5,
	ymin=0,ymax=2.500,
        ybar=2*\pgflinewidth,
    bar width=3pt,
	xtick={1,2,3,4,5,6,7,8,9},
	ytick={0,.5,1,1.5,2},
	legend pos=north west,
	ymajorgrids=true,
	grid style=dashed,
] 
\addplot table[x=N, y=V] from \GFLOPSNodal;
\addplot table[x=N, y=S] from \GFLOPSNodal;
\legend{Volume, Surface}
\end{axis}
\end{tikzpicture}
}
\subfloat{
\begin{tikzpicture}
\begin{axis}[
	width=.36\textwidth,
	legend cell align=left,
	legend style={font=\tiny},
	legend image post style={scale=0.5},
	title={Performance (EPT nodal)},
	xlabel={Degree $N$},
	xmin=.5, xmax=9.5,
	ymin=0,ymax=2.500,
        ybar=2*\pgflinewidth,
    bar width=2pt,
	xtick={1,2,3,4,5,6,7,8,9},
	ytick={0,.5,1,1.5,2},	
	legend pos=north west,
	ymajorgrids=true,
	grid style=dashed,
] 
\addplot table[x=N, y=V] from \GFLOPSBlockNodal;
\addplot table[x=N, y=S] from \GFLOPSBlockNodal;
\addplot table[x=N, y=F] from \GFLOPSBlockNodal;
\legend{Volume, Lift, Flux}

\end{axis}
\end{tikzpicture}
}
\subfloat{
\begin{tikzpicture}
\begin{axis}[
	width=.36\textwidth,
	legend cell align=left,
	legend style={font=\tiny},
	legend image post style={scale=0.5},
	title={Performance (Bernstein)},
	xlabel={Degree $N$},
	xmin=.5, xmax=9.5,
	ymin=0,ymax=2.5,
        ybar=2*\pgflinewidth,
    bar width=3pt,
	xtick={1,2,3,4,5,6,7,8,9},
	ytick={0,.5,1,1.5,2},	
	legend pos=north west,
	ymajorgrids=true,
	grid style=dashed,
] 
\addplot table[x=N, y=V] from \GFLOPSBern;
\addplot table[x=N, y=S] from \GFLOPSBern;
\addplot table[x=N, y=L] from \GFLOPSBern;
\legend{Volume, Surface, Surface (optimal)}
\end{axis}
\end{tikzpicture}
}
\caption{Achieved computational performance (TFLOPS/s) for volume and surface kernels using nodal and Bernstein-Bezier bases.}
\label{fig:gflops}
\end{figure}
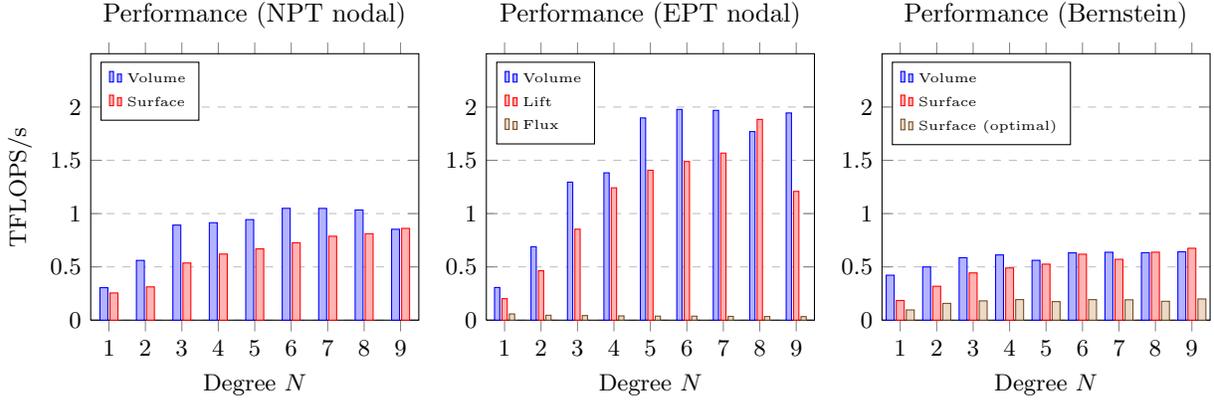


\begin{figure}
\centering
\subfloat{
\begin{tikzpicture}
\begin{axis}[
	width=.355\textwidth,
	legend cell align=left,
	legend style={font=\tiny},
	legend image post style={scale=0.5},	
	title={Bandwidth (NPT nodal)},
	xlabel={Degree $N$},
	ylabel={GB/s},
	xmin=.5, xmax=9.5,
	ymin=0,ymax=275,
        ybar=2*\pgflinewidth,
    bar width=3pt,
	xtick={1,2,3,4,5,6,7,8,9},
	legend pos=north east,
	ymajorgrids=true,
	grid style=dashed,
] 
\addplot table[x=N, y=V] from \BWNodal;
\addplot table[x=N, y=S] from \BWNodal;
\legend{Volume, Surface}

\end{axis}
\end{tikzpicture}
}
\subfloat{
\begin{tikzpicture}
\begin{axis}[
	width=.355\textwidth,
	legend cell align=left,
	legend style={font=\tiny},
	legend image post style={scale=0.5},
	title={Bandwidth (EPT nodal)},
	xlabel={Degree $N$},
	xmin=.5, xmax=9.5,
	ymin=0,ymax=275,
        ybar=2*\pgflinewidth,
    bar width=1.75pt,
	xtick={1,2,3,4,5,6,7,8,9},
	legend pos=north east,
	ymajorgrids=true,
	grid style=dashed,
] 
\addplot table[x=N, y=V] from \BWBlockNodal;
\addplot table[x=N, y=S] from \BWBlockNodal;
\addplot table[x=N, y=F] from \BWBlockNodal;
\legend{Volume, Lift, Flux}

\end{axis}
\end{tikzpicture}
}
\subfloat{
\begin{tikzpicture}
\begin{axis}[
	width=.355\textwidth,
	legend cell align=left,
	legend style={font=\tiny},
	legend image post style={scale=0.5},	
	title={Bandwidth (Bernstein)},
	xlabel={Degree $N$},
	xmin=.5, xmax=9.5,
	ymin=0,ymax=275,
        ybar=2*\pgflinewidth,
    bar width=3pt,
	xtick={1,2,3,4,5,6,7,8,9},
	legend pos=north east,
	ymajorgrids=true,
	grid style=dashed,
] 

\addplot table[x=N, y=V] from \BWBern;
\addplot table[x=N, y=S] from \BWBern;
\addplot table[x=N, y=L] from \BWBern;
\legend{Volume, Surface, Surface (optimal)}
\end{axis}
\end{tikzpicture}
}
\caption{Achieved bandwidth (GB/s) for volume and surface kernels using nodal and Bernstein-Bezier bases.}
\label{fig:bw}
\end{figure}
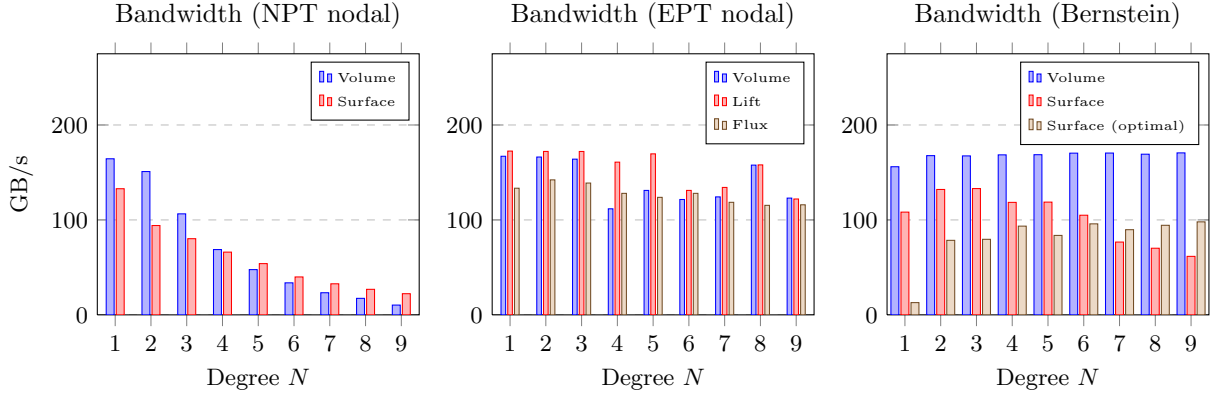


A convenient visualization of computational performance is the Roofline model \cite{williams2009roofline}, which sets an upper bound on the rate of floating point operations (FLOPS/s) based on the arithmetic intensity of a given kernel.  The arithmetic intensity is defined as the work done per unit of data loaded
\[
\text{\rm arithmetic intensity} = \frac{\text{\rm FLOPs performed}}{\text{\rm bytes loaded}}.  
\]
The roofline then produces an upper bound based on the theoretically \textit{attainable} performance, defined as
\[
\text{\rm attainable performance} = \min{\LRp{\text{\rm arithmetic intensity} \times \text{\rm peak bandwidth}, \text{\rm peak GFLOPS/s}}}.
\]
This bound indicates that the peak attainable performance increases with arithmetic intensity until peak computational performance is reached, at which point the roofline flattens out.  Plotting the achieved performance (GFLOPS per second) of each kernel against the arithmetic intensity (GFLOPS per GB of data loaded) results in a point which lies somewhere below the roofline bound.  The distance of each point to the roofline is indicative of performance ``ceilings'', which prevent peak performance due to the structure of the kernel.  

The two extremes of performance lie at the bottom left and top right of the roofline plot.  The bottom left of the roofline bound represents the bandwidth or data-bound regime, where a small number of computations are performed for each byte of data loaded.  The top right of the roofline (where the bound flattens out) indicates that a kernel is compute or FLOP-bound, and that a large number of computations are performed for each data load.  The approach towards kernel optimization depend on whether a kernel is bandwidth-bound or compute-bound.  

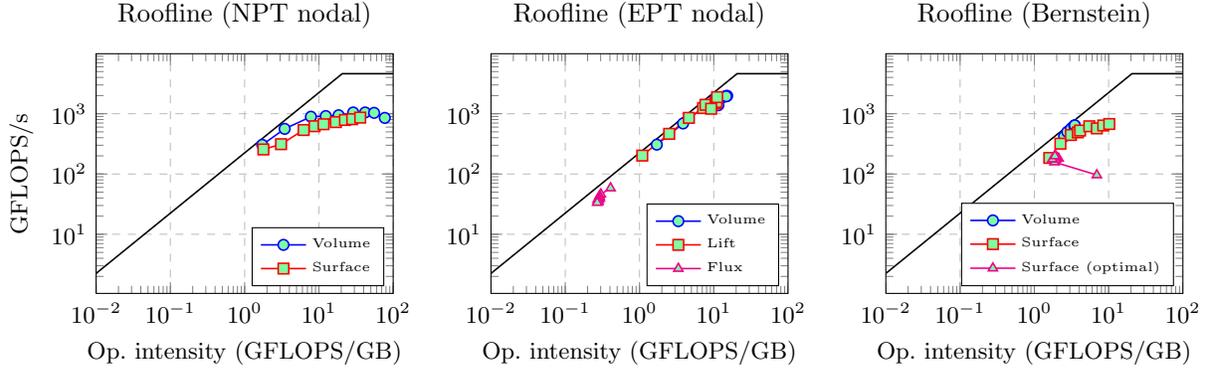
\begin{figure}
\centering
\subfloat{
\begin{tikzpicture}
\begin{loglogaxis}[
	legend cell align=left,
	legend style={font=\tiny},
	width=.335\textwidth,
    title={Roofline (NPT nodal)},
    xlabel={Op.\ intensity (GFLOPS/GB)},
    ylabel={GFLOPS/s},
    xmin=.01, xmax=100,
    legend pos=south east,
    xmajorgrids=true,
    ymajorgrids=true,
    grid style=dashed,
] 
\addplot+[color=blue,mark=*,mark options={fill=markercolor},semithick]
coordinates{(1.72666,304.8)(3.45915,560.7)(7.80756,891.9)(12.3847,914.5)(18.4345,943.3)(28.9758,1049.3)(41.9139,1048.7)(55.3697,1033.8)(77.1189,852.9)};
\addplot+[color=red,mark=square*,mark options={fill=markercolor},semithick]
coordinates{(1.78876,255.179)(3.09065,312.277)(6.2277,536.494)(8.74391,620.34)(11.5419,668.915)(16.9289,725.89)(22.4189,787.639)(28.0905,811.189)(35.9825,861.195)};
\addplot+[color=black,semithick,mark=none]
coordinates{(0.01,2.24)(20.5893,4612)(100,4612)};
\legend{Volume, Surface,}
\end{loglogaxis}
\end{tikzpicture}
}
\subfloat{
\begin{tikzpicture}
\begin{loglogaxis}[
	legend style={font=\tiny},
	legend cell align=left,
	width=.335\textwidth,
    title={Roofline (EPT nodal)},
    xlabel={Op.\ intensity (GFLOPS/GB)},
    xmin=.01, xmax=100,
    legend pos=south east,
    xmajorgrids=true,
    ymajorgrids=true,
    grid style=dashed,
] 
\addplot+[color=blue,mark=*,mark options={fill=markercolor},semithick]
coordinates{(1.71001,306.71)(3.86018,689.05)(7.35216,1294.7)(11.525,1382.4)(13.4836,1897.6)(15.1548,1977.4)(14.7517,1968.3)(10.4485,1769.5)(14.724,1944.3)};
\addplot+[color=red,mark=square*,mark options={fill=markercolor},semithick]
coordinates{(1.08977,201.86)(2.50909,463.74)(4.61984,853.66)(7.18954,1241.7)(7.72459,1406.6)(10.577,1488.7)(10.8768,1567.5)(11.1056,1883.6)(9.22081,1208.3)};
\addplot+[color=magenta,mark=triangle*,mark options={fill=markercolor},semithick]
coordinates{(0.407665,58.3849)(0.30195,46.1)(0.289641,43.1795)(0.292358,40.1741)(0.28497,37.8729)(0.272239,37.411)(0.278051,35.3907)(0.27915,34.575)(0.269161,33.5072)};

\addplot+[color=black,semithick,mark=none]
coordinates{(0.01,2.24)(20.5893,4612)(100,4612)};
\legend{Volume, Lift, Flux,}
\end{loglogaxis}
\end{tikzpicture}
}
\subfloat{
\begin{tikzpicture}
\begin{loglogaxis}[
	legend style={font=\tiny},
	legend cell align=left,
	width=.335\textwidth,
    title={Roofline (Bernstein)},
    xlabel={Op.\ intensity (GFLOPS/GB)},
    xmin=.01, xmax=100,
    legend pos=south east,
    xmajorgrids=true,
    ymajorgrids=true,
    grid style=dashed,
]
\addplot+[color=blue,mark=*,mark options={fill=markercolor},semithick]
coordinates{(2.51771,421.801)(2.77537,499.781)(3.25934,585.95)(3.3832,612.231)(3.10056,561.642)(3.45345,631.582)(3.48214,637.221)(3.48439,633.11)(3.5012,641.385)};

\addplot+[color=red,mark=square*,mark options={fill=markercolor},semithick]
coordinates{(1.58974,184.714)(2.24325,318.13)(3.10412,443.399)(3.85037,489.824)(4.12215,525.607)(5.49369,619.05)(6.93651,571.032)(8.46519,637.878)(10.1983,674.671)};

\addplot+[color=magenta,mark=triangle*,mark options={fill=markercolor},semithick]
coordinates{(6.93491,96.102)(1.86777,157.412)(2.11827,180.978)(1.92284,192.954)(1.93925,174.241)(1.8711,192.588)(1.98199,190.916)(1.74658,176.906)(1.89114,198.934)};

\addplot+[color=black,semithick,mark=none]
coordinates{(0.01,2.24)(20.5893,4612)(100,4612)};

\legend{Volume, Surface, Surface (optimal)}

\end{loglogaxis}
\end{tikzpicture}
}
\caption{Roofline performance analysis for nodal and Bernstein-Bezier bases for $N = 1,\ldots,9$.}
\label{fig:roofline}
\end{figure}



Figure~\ref{fig:roofline} presents roofline plots of the volume and surface kernels for nodal and Bernstein-Bezier bases at orders of approximation $N = 1,\ldots, 9$.  For nodal bases, the achieved computational performance of both the volume and surface kernel stalls at higher orders.   This is alleviated for the EPT nodal kernels, where manual management of fast memory greatly improves performance at higher $N$ relative to the NPT nodal kernels.  We note that, while the EPT nodal kernels perform the same number of operations as the nodal kernels, the operational intensity differs due to the fact that the recorded bandwidth is increased by loading some data redundantly.  

For the Bernstein-Bezier basis, the arithmetic intensity is decreased at higher polynomial degrees by sparsifying the derivative and lift operators.  As a result, the Bernstein derivative kernel delivers roughly the same computational performance and operational intensity independently of polynomial order.  The performance of the Bernstein surface kernel is similar to that of the nodal surface kernel, but with similar performance observed at lower operational intensities.  The optimal Bernstein surface kernel behaves similarly to the Bernstein derivative kernel, with the ratio of computational performance and operational intensity remaining relatively constant for $N > 1$.  The outlying roofline point for the optimal Bernstein surface kernel in Figure~\ref{fig:roofline} corresponds to $N=1$.  

Since the roofline provides a tighter bound on attainable performance than the peak performance (which is typically only achieved under very specific computational scenarios), we can also use the ratio between observed and attainable performance to measure the percentage of computational efficiency achieved.  As noted in \cite{axelGPU2015}, this percentage is equal to the ratio of achieved bandwidth to peak bandwidth for bandwidth-bound kernels.  Figure~\ref{fig:efficiency} shows the percentage of computational efficiency of volume and surface kernels under each basis.  The efficiency of the NPT nodal volume and surface kernels decreases to around $20\%$ at large $N$;  in comparison, the efficiency of the EPT nodal kernels maintain a higher efficiency around $60\%$ (with the exact percentage fluctuating depending on the order).  In comparison, the Bernstein volume kernel maintains a high percentage of efficiency between $75-80\%$ for the volume kernel.   The Bernstein surface kernel does appear to lose efficiency at higher orders; however, this degradation is delayed until around $N=5$, whereas the NPT nodal surface kernel shows efficiency loss for $N>1$.  In contrast, the efficiency of the optimal Bernstein surface kernel increases constantly with $N$, remaining around $40\%$ for all $N > 1$.  

\begin{figure}
\centering
\subfloat{
\begin{tikzpicture}
\begin{axis}[
 	legend cell align=left,
	legend style={font=\tiny},
	width=.33\textwidth,
	title={Comp.\ efficiency (NPT)},
	xlabel={Degree $N$},
	ylabel={\% of attainable perf.},
	xmin=.5, xmax=9.5,
	ymin=-10, ymax=150,
	ytick = {0,20,40,60,80,100},
	xtick={1,2,3,4,5,6,7,8,9},	
	yticklabel=\pgfmathparse{\tick}\pgfmathprintnumber{\pgfmathresult}\,\%,
	legend pos=north east,
	xmajorgrids=true,
	ymajorgrids=true,
	grid style=dashed,
] 
\addplot+[color=blue,mark=*,semithick,mark options={fill=markercolor}]
coordinates{(1,73.3937)(2,67.3929)(3,47.4955)(4,30.7009)(5,21.275)(6,22.7515)(7,22.7385)(8,22.4154)(9,18.4931)};
\addplot+[color=red,mark=square*,semithick,mark options={fill=markercolor}]
coordinates{(1,59.3125)(2,42.0089)(3,35.817)(4,29.4969)(5,24.096)(6,17.8277)(7,17.078)(8,17.5887)(9,18.6729)};
\legend{Volume, Surface}
\end{axis}
\end{tikzpicture}
}
\subfloat{
\begin{tikzpicture}
\begin{axis}[
 	legend cell align=left,
	legend style={font=\tiny},
	width=.33\textwidth,
	title={Comp.\ efficiency (EPT)},
	xlabel={Degree $N$},
	xmin=.5, xmax=9.5,
	ymin=-10, ymax=150,
	ytick = {0,20,40,60,80,100},
	xtick={1,2,3,4,5,6,7,8,9},	
	yticklabel=\pgfmathparse{\tick}\pgfmathprintnumber{\pgfmathresult}\,\%,
	legend pos=north east,
	xmajorgrids=true,
	ymajorgrids=true,
	grid style=dashed,
] 
\addplot+[color=blue,mark=*,semithick,mark options={fill=markercolor}]
coordinates{(1,74.5728)(2,74.2156)(3,73.2161)(4,49.8705)(5,58.5129)(6,54.2496)(7,55.4754)(8,70.4125)(9,54.9022)};
\addplot+[color=red,mark=square*,semithick,mark options={fill=markercolor}]
coordinates{(1,77.0134)(2,76.8442)(3,76.8263)(4,71.8071)(5,75.7089)(6,58.5188)(7,59.9183)(8,70.5179)(9,54.4826)};
\addplot+[color=magenta,mark=triangle*,semithick,mark options={fill=markercolor}]
coordinates{(1,59.5455)(2,63.4772)(3,61.9826)(4,57.1326)(5,55.2563)(6,57.1348)(7,52.9196)(8,51.4964)(9,51.758)};

\legend{Volume, Lift, Flux}
\end{axis}
\end{tikzpicture}
}
\subfloat{
\begin{tikzpicture}
\begin{axis}[
	legend cell align=left,
	legend style={font=\tiny},	
	width=.33\textwidth,
    title={Comp.\ efficiency (Bernstein)},
    xlabel={Degree $N$},
    xmin=.5, xmax=9.5,
    ymin=-10, ymax=150,
    ytick = {0,20,40,60,80,100},
	xtick={1,2,3,4,5,6,7,8,9},    
    yticklabel=\pgfmathparse{\tick}\pgfmathprintnumber{\pgfmathresult}\,\%,
 legend pos=north east,
    xmajorgrids=true,   
    ymajorgrids=true,    
    grid style=dashed,
] 

\addplot+[color=blue,mark=*,semithick,mark options={fill=markercolor}]
coordinates{(1,69.6554)(2,74.8705)(3,74.7451)(4,75.2384)(5,75.3134)(6,76.0375)(7,76.0844)(8,75.5446)(9,76.1647)};

\addplot+[color=red,mark=square*,semithick,mark options={fill=markercolor}]
coordinates{(1,48.3089)(2,58.9629)(3,59.3893)(4,52.892)(5,53.0138)(6,46.8504)(7,34.2272)(8,31.3295)(9,27.5054)};

\addplot+[color=magenta,mark=triangle*,semithick,mark options={fill=markercolor}]
coordinates{(1,5.76161)(2,35.0402)(3,35.5219)(4,41.7219)(5,37.3567)(6,42.7942)(7,40.0491)(8,42.1121)(9,43.7357)};

\legend{Volume, Surface, Surface (optimal)}

\end{axis}
\end{tikzpicture}
}
\caption{Computational efficiency (percentage of attainable performance based on the roofline bound) of volume and surface kernels for nodal and Bernstein-Bezier bases.}
\label{fig:efficiency}
\end{figure}
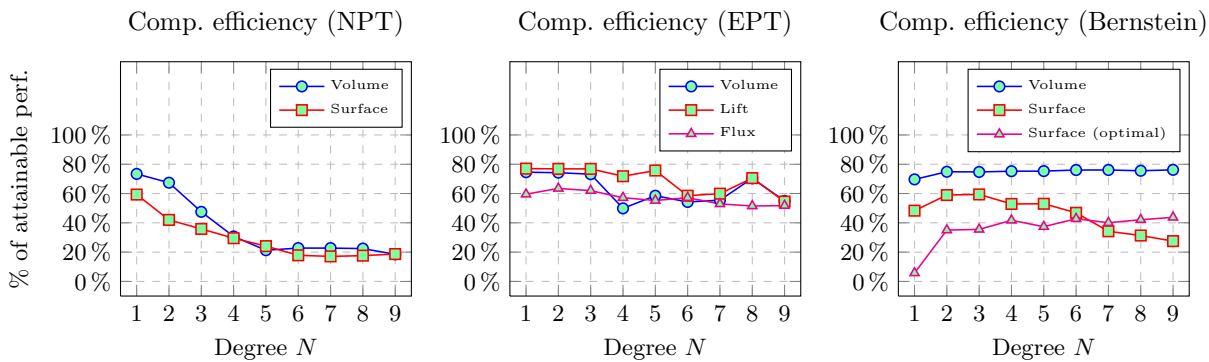



\section{Conclusions and future work}

We have presented a comparison of GPU-accelerated implementations of DG using node-per-thread (NPT) matrix-multiplication kernels, element-per-thread (EPT) blocked matrix-multiplication kernels, and specialized kernels which take advantage of properties of the Bernstein-Bezier basis.  The efficient application of sparse Bernstein derivative matrices is discussed, and a factorization of the lift matrix is presented which allows for both faster application and more numerically stable computations.  

Two methods for applying the Bernstein-Bezier lift matrix are also presented.   The first method takes advantage of the sparsity of the factorized lift matrix and is faster $N\leq 6$, while the second method results in asymptotically optimal complexity and improved performance at very high orders of approximation.  Computational results demonstrate the degradation of performance for straightforward implementations of the volume and surface kernels at high orders of approximation, and suggest that high order Bernstein-Bezier bases have the potential to provide an efficient alternative to high order nodal bases for DG time domain methods.  GPU-accelerated implementations of Bernstein-Bezier DG are observed to yield $2\times$ speedups over straightforward implementations of nodal DG at orders $N = 4,5$, which improves to a $5\times$ speedup at $N = 9$.  

We note that the Bernstein-Bezier basis achieves the largest speedup over the nodal basis at very high orders of approximation.  However, additional barriers exist which reduce the efficiency of very high order time-explicit DG methods, the most evident of which is the CFL condition.  It is well-known that, for a mesh size $h$ and polynomial degree $N$, the stable timestep restriction for upwind DG scales as
\[
dt \leq C \frac{h}{N^2},
\]
where $C$ is some constant which depends on the physical wavespeed.  Some advances have been made in addressing the restrictiveness of this condition.  Warburton and Hagstrom introduced an filter implicitly defined by transferring the solution to and from a dual grid \cite{warburton2008taming}.  Under such a filter, the stable timestep restriction can be shown to be improved to
\[
dt \leq C\frac{h}{N}.
\]
Reyna and Li showed that the central DG method results in a similar timestep restriction \cite{reyna2014operator}.  Xu, Chen, and Liu introduced additional conservation constraints (enforced through either Lagrange multipliers or penalization) based on propagation of the solution average forward in time \cite{xu2014new}.  The incorporation of such methods will also be necessary to alleviate the onerous timestep restriction at very high orders of approximation.  

The work of the first author (JC) was supported partially by TOTAL E\&P Research and Technology USA, contract number 20150023.  The second author (TW) was supported partially by ANL (award number 1F-32301, subcontract on DOE DE-AC02-06CH11357). Both JC and TW would like to acknowledge the support of NSF (award number DMS-1216674) in this research, as well as Mark Ainsworth and Rajesh Gandham for helpful discussions.  

\appendix

\section{Properties of Bernstein-Bezier operators}
\label{app:app}

Bernstein-Bezier polynomials in $d$ dimensions are typically described using a multi-index notation.  For example, for $d=3$, a basis function $B^N_{ijkl}(\lambda_1,\lambda_2,\lambda_3,\lambda_4)$ is notated as $B^N_{\bm{\alpha}}(\bm{\lambda})$, where $\bm{\alpha}, \bm{\lambda} \in \bm{R}^{d+1}$. In three dimensions, the multi-index $\bm{\alpha} = \LRp{\alpha_0,\ldots,\alpha_{d}}$ corresponds to the tuple of barycentric indices $(i,j,k,l)$ and $\bm{\lambda}$ corresponds to the vector of barycentric coordinates $\LRp{\lambda_1,\ldots,\lambda_4}$.  We introduce notation for the comparison and summation of multi-indices as well
\begin{align*}
\LRb{\bm{\alpha}}&\coloneqq \sum_{j=0}^d \alpha_j,\\
\bm{ \alpha} &\leq \bm{\beta} \quad \text{\rm if } \alpha_i \leq \beta_i, \quad \forall i = 0,\ldots, d.  
\end{align*}
Under this notation, Bernstein polynomials have a very compact dimension-independent representation in terms of barycentric coordinates.  

We will also refer to hierarchical orthogonal (modal) polynomials using a similar multi-index notation.  Let $\bm{r}\in \mbb{R}^d$ be the vector of coordinates on the reference tetrahedron $\widehat{D}$; then, we introduce a hierarchical orthogonal basis ${L}_{\bm{\gamma}}(\bm{r})$ with multi-index $\bm{\gamma}\in\mbb{R}^{d}$ such that 
\begin{align*}
\int_{\widehat{D}} {L}_{\bm{\nu}}(\bm{r}) {L}_{\bm{\gamma}}(\bm{r}) &= 0, \quad \bm{\nu} \neq \bm{\gamma}\\
{L}_{\bm{\gamma}}(\bm{r}) &\in P^N(\widehat{D}), \quad \LRb{\bm{\gamma}} \leq N.
\end{align*}
For convenience, we also use multi-index notation to refer to entries of matrices and vectors associated with Bernstein and modal polynomials.  

Bernstein polynomials yield two useful properties.  The first is that degree elevation is sparse; in other words, a Bernstein polynomial of degree $N-1$ may be represented using no more than $d+1$ Bernstein polynomials of degree $N$.  Degree elevation of a Bernstein polynomial $B^{N-1}_{\bm{\alpha}}$ can be written explicitly as
\[
B^{N-1}_{\bm{\alpha}} = \sum_{j=0}^d \frac{\alpha_j+1}{N}B^N_{\bm{\alpha} + \bm{e}_j}.
\]
where $\bm{e}_j$ is the $j$th canonical vector.  The second useful property of Bernstein polynomials is the sparsity of differentiation.   The derivative of a Bernstein polynomial $B^N_{\bm{\alpha}}$ with respect to a barycentric coordinate $\lambda_i$ is $NB^{N-1}_{\bm{\alpha}-\bm{e}_i}$.  For example, if $i=0$
\[
\pd{B^N_{\bm{\alpha}}}{\lambda_0}{} = \pd{}{\lambda_0}{}\frac{N!}{i!j!k!l!}\lambda_0^i\lambda_1^j\lambda_2^k\lambda_3^l = N\frac{(N-1)!}{(i-1)!j!k!l!} \lambda_0^{i-1}\lambda_1^j \lambda_2^k\lambda_3^l = NB^{N-1}_{\bm{\alpha}-\bm{e}_0}.
\]
These are used to show properties of derivative and lift operators for Bernstein polynomials.  

\subsection{Derivative operators}
\label{app:bb}

Combining degree elevation and differentiation formulas gives the sparsity of the barycentric derivative matrices $\bm{D}^0,\ldots, \bm{D}^3$ explicitly: 
\begin{lemma}
Assume $\bm{D}^i$ is the derivative operator w.r.t.\ $i$th barycentric coordinate, and that $\bm{\alpha},\bm{\beta}$ are multi-indices. For a fixed column $\bm{\beta}$, the non-zero row entries of $\bm{D}^i$ have indices $\bm{\alpha}$ such that
\[
\bm{D}^i_{\bm{\alpha},\bm{\beta}} = \alpha_j+1, \qquad \bm{\alpha} = \bm{\beta} - \bm{e}_i + \bm{e}_j, \qquad j = 0,\ldots,d
\]
For a fixed row $\alpha$, the non-zero column entries of $D^i$ have indices $\beta$ such that
\[
\bm{D}^i_{\bm{\alpha},\bm{\beta}} = \alpha_j, \qquad \bm{\beta} = \bm{\alpha} + \bm{e}_i - \bm{e}_j, \qquad j = 0,\ldots,d.
\]
\end{lemma}
\begin{proof}
A column of $\bm{D}^i$ is characterized by the representation of the derivative of $B^N_{\bm{\alpha}}$ as a linear combination of Bernstein polynomials of the same degree.  By the above formulas for degree elevation and the barycentric derivative of $B^N_{\bm{\alpha}}$, 
\[
\pd{}{\lambda_i}{}B^N_{\bm{\alpha}} = N B^{N-1}_{\bm{\alpha}-\bm{e}_i} = \sum_{j=0}^d (\alpha_j+1)B^N_{\bm{\alpha} -\bm{e}_i + \bm{e}_j}
\]
which gives the row indices $\bm{\beta}$ for a given column $\bm{\alpha}$ of $\bm{D}^i$. The column indices $\bm{\beta}$ of a row with index $\bm{\alpha}$ are straightforward to determine based on this relation.  
\end{proof}

\subsection{Degree reduction operators}

We show here some useful properties of the degree reduction operator, which are used to derive the structure of the Bernstein lift matrix.  We will make use of the fact that a Bernstein mass matrix of degree $N$ in any dimension contains only $(N+1)$ distinct eigenvalues $\lambda^N_i$, whose explicit expressions are given in \cite{derriennic1985multivariate}.  The multiplicity of each eigenvalue corresponds to the dimension of the eigenspace, each of which consists of polynomials orthogonal to all polynomials of degree $i-1$ for $0< i \leq N$.  

\begin{lemma}
\label{lemma:elev}
Suppose $p \in P^N(\widehat{D})$.  Let $\bm{T}$ be the transformation matrix such that
\[
p = \sum_{\LRb{\bm{\gamma}}\leq N} c^L_{\bm{\gamma}} L_{\bm{\gamma}} = \sum_{\LRb{\bm{\alpha}}= N} c^B_{\bm{\alpha}} B^N_{\bm{\alpha}}, \qquad \bm{c}^B = \bm{T} \bm{c}^L
\]
where $L_{\bm{\gamma}}, B^N_{\bm{\alpha}}$ are modal and Bernstein polynomials, respectively.  Then, 
\[
\bm{T}_{N-i}^{-1}\LRp{\bm{E}_{N-i}^N}^T\bm{T}_{N} =  \bm{\tilde{D}}.
\]
Suppose  $0\leq k \leq N$, and let $\lambda^N_k, \lambda^{N-i}_k$ be the distinct eigenvalues of $\bm{M}_N$ and $\bm{M}_{N-i}$, respectively. The entries of $\tilde{\bm{D}}$ are given as
\[
\bm{\tilde{D}}_{\bm{\nu},\bm{\gamma}} = \begin{cases}
{\lambda_{\LRb{\bm{\gamma}}}^{N-i}}/{\lambda^N_{\LRb{\bm{\gamma}}}},  &\bm{\nu} = \bm{\gamma}\\
0  &\text{\rm otherwise}
\end{cases} \qquad \bm{\tilde{D}} \in \mbb{R}^{\LRp{N-i}_p,N_p}
\]
where $N_p, \LRp{N-i}_p$ are the dimensions of the space of polynomials of total degree $N$ and $(N-i)$, respectively.
\end{lemma}
\begin{proof}
The degree elevation operator satisfies
\[
\LRp{\bm{E}^N_{N-i}}^T \bm{M}_{N,N} = \bm{M}_{N-i,N},
\]
where we have dropped the dimension-dependent superscripts on $\bm{M}_{N,N}^d$ for clarity.  Rearrangement and multiplying by changes of basis from each side gives
\[
\bm{T}_{N-i}^{-1} \LRp{\bm{E}_{N-i}^N}^T \bm{T}_N = \LRp{\bm{T}_{N-i}^{-1}\bm{M}_{N-i,N} \bm{T}_N} \LRp{\bm{T}_N^{-1}\bm{M}_N^{-1}\bm{T}_N}.
\]
Since $L_{\bm{\gamma}}$ are orthogonal with respect to the $L^2$ inner product induced by the mass matrix, $\bm{M}_N$ is a diagonal matrix under the change of basis $\bm{T}_{N}$, with entries 
\[
\LRp{\bm{T}_{N}^{-1}\bm{M}^{-1}_{N,N} \bm{T}_N}_{\bm{\gamma},\bm{\gamma}} = 1/\lambda^N_{\LRb{\bm{\gamma}}}.  
\]
where $\lambda^N_{\LRb{\bm{\gamma}}}$ is the $\LRb{\bm{\gamma}}$th distinct eigenvalue of $M_{N,N}$.  By similar orthogonality arguments, $\LRp{\bm{T}_{N-i}^{-1}\bm{M}_{N-i,N} \bm{T}_N}$ is a rectangular matrix with entries
\[
\LRp{\bm{T}_{N-i}^{-1}\bm{M}_{N-i,N} \bm{T}_N}_{\bm{\nu},\bm{\gamma}} = \begin{cases}
0, & \bm{\nu} \neq \bm{\gamma}\\
\lambda^{N-i}_{\LRb{\bm{\gamma}}}, &\bm{\nu} = \bm{\gamma}
\end{cases}
\]
\end{proof}
This result implies that, while degree reduction does not preserve modes exactly, it preserves separation between the modes.  This is simplest to see in 1D, where transforming the degree reduction operator to a modal basis results in a rectangular diagonal matrix with more columns than rows.  Straightforward computations give the following corollary:
\begin{corollary}
Under a transformation to a modal (orthogonal polynomial) basis, $\bm{E}^N_{N-i}\LRp{\bm{E}^N_{N-i}}^T$ is diagonal, with entries
\[
\LRp{\bm{T}^{-1}_N\bm{E}^N_{N-i}\LRp{\bm{E}^N_{N-i}}^T\bm{T}_N}_{\bm{\gamma},\bm{\gamma}} = \begin{cases}
0,  & \LRb{\bm{\gamma}} > (N-i),\\
\lambda^{N-i}_{\LRb{\bm{\gamma}}}/\lambda^N_{\LRb{\bm{\gamma}}}, &\LRb{\bm{\gamma}} \leq (N-i).
\end{cases}
\] 
\end{corollary}
In other words, the composition of degree reduction and elevation results only in the scaling or truncation of individual polynomial modes.  

\subsection{Lift matrices}
\label{app:lift}

We observe that the Bernstein-Bezier lift matrix for a single face $f$ on a $d$-dimensional simplex is of the form
\[
\bm{L} = \bm{M}^{-1}\bm{M}_f = \LRs{
\begin{array}{c}
\ell_0 \bm{L}_0\\
\ell_1 \LRp{\bm{E}^N_{N-1}}^T\bm{L}_0\\
\vdots\\
\ell_N \LRp{\bm{E}^N_{0}}^T\bm{L}_0
\end{array},
}
\]
where $\bm{E}_{N-i}^N$ is the degree elevation operator from degree $N-i$ polynomials on a $(d-1)$ dimensional simplex, and 
\[
\bm{L}_0 = \frac{(N+1)^2}{2} \LRp{\bm{E}^{N+1}_N}^T \bm{E}^{N+1}_N
\]
and the constants $\ell_1,\ldots,\ell_N$ have the explicit representation
\[
\ell_j = (-1)^j \frac{\binom{N}{j}}{1+j}.
\]
Coincidentally, $\ell_j$ are also the coefficients of the Bernstein-Bezier representation of the one-dimensional Lagrange polynomial $\ell_{\rm GRJ}$ which is unity at $r = -1$ and zero at the remaining $N$ GRJ-$(0,0)$ quadrature points.  These polynomials relate strongly to $hp$ finite element trace constants on simplices, which are derived as extremal eigenvalues of a generalized eigenvalue problem.  The corresponding extremal eigenfunction is constructed using $\ell_{\rm GRJ}$ \cite{warburton2003constants}.



We also observe connections between $\bm{L}_0$ and $hp$ trace constants for the face of a $d$-dimensional simplex.  These constants are related to the lift matrix through the generalized eigenvalue problem involving the face mass matrix $\bm{M}_f$, the mass matrix $\bm{M}$.  
\begin{lemma}
The eigenvalues of $\bm{L}_0$ are scalings of the non-zero eigenvalues $\lambda_i$ of the generalized eigenvalue problem
\[
\bm{M}_f\bm{u} = \lambda_i \bm{M} \bm{u}.
\]
\end{lemma}
\begin{proof}
Expressions for the generalized eigenvalues, 
\[
\lambda_i = \frac{(N+i+d)(N+1-i)}{2}, \qquad i = 0,\ldots, N,
\]
can be explicitly derived for a $d$-dimensional simplex from expressions in \cite{warburton2003constants} for $\bm{M}_f$ under a modal basis. 

Kirby \cite{kirby2015efficient}, Farouki \cite{farouki2003construction}, and Derrenic \cite{derriennic1985multivariate} derive that the unique eigenvalues of the degree $N$ reference Bernstein mass matrix in $d$ dimensions are 
\[
\lambda^N_i = \LRb{\widehat{D}} \frac{(N!)^2d!}{(N+i+d)!(N-i)!}.
\]
Recall that, for $n \leq N$, under the transformation $\bm{T}_N$ to a hierarchical orthogonal (modal) polynomial basis,
\[
\bm{T}_N \LRp{E^N_i}^T {\bm{T}_n}^{-1} = {\rm diag}\LRp{\frac{\lambda^i_0}{\lambda^N_0},\ldots,\frac{\lambda^i_i}{\lambda^N_{i}},0,\ldots,0}.
\]
This implies that $\bm{L}_0 = \frac{(N+1)^2}{2} \LRp{\bm{E}^{N+1}_N}^T \bm{E}^{N+1}_N$ has eigenvalues
\[
\lambda(\bm{L}_0) = \frac{(N+1)^2}{2} \LRp{\frac{\lambda_0^N}{\lambda^{N+1}_0}}, \ldots, \frac{(N+1)^2}{2} \LRp{\frac{\lambda_N^N}{\lambda^{N+1}_N}}
\]
where $\lambda^N_i$ are eigenvalues for degree $N$ Bernstein mass matrices in dimension $d-1$.  The proof is completed by noting that the formula for $\frac{\lambda_i^N}{\lambda^{N+1}_i}$ reduces to 
\[
\frac{\lambda_i^N}{\lambda^{N+1}_i} = \frac{(N+i+d)(N+1-i)}{(N+1)^2}.
\]

\end{proof}


\bibliographystyle{unsrt}
\bibliography{ipdg}

\begin{thebibliography}{10}

\bibitem{klockner2009nodal}
Andreas Kl{\"o}ckner, Tim Warburton, Jeff Bridge, and Jan~S Hesthaven.
\newblock Nodal discontinuous {Galerkin} methods on graphics processors.
\newblock {\em Journal of Computational Physics}, 228(21):7863--7882, 2009.

\bibitem{hesthaven2007nodal}
Jan~S Hesthaven and Tim Warburton.
\newblock {\em Nodal discontinuous Galerkin methods: algorithms, analysis, and
  applications}, volume~54.
\newblock Springer, 2007.

\bibitem{fatahalian2004understanding}
Kayvon Fatahalian, Jeremy Sugerman, and Pat Hanrahan.
\newblock Understanding the efficiency of {GPU} algorithms for matrix-matrix
  multiplication.
\newblock In {\em Proceedings of the ACM SIGGRAPH/EUROGRAPHICS conference on
  Graphics hardware}, pages 133--137. ACM, 2004.

\bibitem{volkov2008benchmarking}
Vasily Volkov and James~W Demmel.
\newblock Benchmarking {GPU}s to tune dense linear algebra.
\newblock In {\em High Performance Computing, Networking, Storage and Analysis,
  2008. SC 2008. International Conference for}, pages 1--11. IEEE, 2008.

\bibitem{nvidia-cublas}
nVidia.
\newblock {\em {CUBLAS Library User Guide}}.
\newblock nVidia, v5.0 edition, October 2012.

\bibitem{farouki2012bernstein}
Rida~T Farouki.
\newblock The {Bernstein} polynomial basis: a centennial retrospective.
\newblock {\em Computer Aided Geometric Design}, 29(6):379--419, 2012.

\bibitem{ainsworth2011bernstein}
Mark Ainsworth, Miangaly~Gaelle Andriamaro, and Oleg Davydov.
\newblock Bernstein-{B}{\'e}zier finite elements of arbitrary order and optimal
  assembly procedures.
\newblock {\em SIAM Journal on Scientific Computing}, 33(6):3087--3109, 2011.

\bibitem{ainsworth2014pyramid}
Mark Ainsworth.
\newblock Pyramid algorithms for {Bernstein--B}{\'e}zier finite elements of
  high, nonuniform order in any dimension.
\newblock {\em SIAM Journal on Scientific Computing}, 36(2):A543--A569, 2014.

\bibitem{kirby2011fast}
Robert~C Kirby.
\newblock Fast simplicial finite element algorithms using {Bernstein}
  polynomials.
\newblock {\em Numerische Mathematik}, 117(4):631--652, 2011.

\bibitem{kirby2012fast}
Robert~C Kirby and Kieu~Tri Thinh.
\newblock Fast simplicial quadrature-based finite element operators using
  {Bernstein} polynomials.
\newblock {\em Numerische Mathematik}, 121(2):261--279, 2012.

\bibitem{kirby2015efficient}
Robert~C Kirby.
\newblock Fast inversion of the simplicial {Bernstein} mass matrix.
\newblock {\em Numerische Mathematik}, pages 1--23, 2016.

\bibitem{warburton1999basis}
TC~Warburton, SJ~Sherwin, and GE~Karniadakis.
\newblock Basis functions for triangular and quadrilateral high-order elements.
\newblock {\em SIAM Journal on Scientific Computing}, 20(5):1671, 1999.

\bibitem{beuchler2006new}
Sven Beuchler and Joachim Schoeberl.
\newblock New shape functions for triangular {$p$-FEM} using integrated jacobi
  polynomials.
\newblock {\em Numerische Mathematik}, 103(3):339--366, 2006.

\bibitem{beuchler2012sparsity}
Sven Beuchler, Veronika Pillwein, Joachim Sch{\"o}berl, and Sabine Zaglmayr.
\newblock {\em Sparsity optimized high order finite element functions on
  simplices}.
\newblock Springer, 2012.

\bibitem{warburton2006explicit}
T~Warburton.
\newblock An explicit construction of interpolation nodes on the simplex.
\newblock {\em Journal of engineering mathematics}, 56(3):247--262, 2006.

\bibitem{wozniak2014gimmik}
Bartosz~D Wozniak, Freddie~D Witherden, Francis~P Russell, Peter~E Vincent, and
  Paul~HJ Kelly.
\newblock {GiMMiK} - generating bespoke matrix multiplication kernels for
  accelerators: Application to high-order computational fluid dynamics.
\newblock {\em Computer Physics Communications}, 202:12--22, 2016.

\bibitem{marco2007fast}
Ana Marco, Jos{\'e}-Javier Mart{\i}, et~al.
\newblock A fast and accurate algorithm for solving {Bernstein-Vandermonde}
  linear systems.
\newblock {\em Linear algebra and its applications}, 422(2):616--628, 2007.

\bibitem{chan2015short}
Jesse Chan and T~Warburton.
\newblock A short note on a {Bernstein-Bezier} basis for the pyramid.
\newblock {\em arXiv preprint arXiv:1508.05609}, 2015.

\bibitem{kiciak2004recursive}
Przemys{\l}aw Kiciak and Ahmed Zidna.
\newblock Recursive de {Casteljau} bisection and rounding errors.
\newblock {\em Computer aided geometric design}, 21(7):683--695, 2004.

\bibitem{ipsen2009numerical}
Ilse~CF Ipsen.
\newblock {\em Numerical matrix analysis: Linear systems and least squares}.
\newblock SIAM, 2009.

\bibitem{farouki2000legendre}
Rida~T Farouki.
\newblock {Legendre-Bernstein} basis transformations.
\newblock {\em Journal of Computational and Applied Mathematics},
  119(1):145--160, 2000.

\bibitem{ainsworth2015computing}
Mark Ainsworth and Manuel~A S{\'a}nchez.
\newblock Computing the {B}ezier control points of the {Lagrangian} interpolant
  in arbitrary dimension.
\newblock {\em arXiv preprint arXiv:1510.09197}, 2015.

\bibitem{carpenter1994fourth}
Mark~H Carpenter and Christopher~A Kennedy.
\newblock Fourth-order $2n$-storage {R}unge-{K}utta schemes.
\newblock Technical Report NASA-TM-109112, NAS 1.15:109112, NASA Langley
  Research Center, 1994.

\bibitem{golub2012matrix}
Gene~H Golub and Charles~F Van~Loan.
\newblock {\em Matrix computations}, volume~3.
\newblock JHU Press, 2012.

\bibitem{nath2010improved}
Rajib Nath, Stanimire Tomov, and Jack Dongarra.
\newblock An improved {MAGMA GEMM for Fermi} graphics processing units.
\newblock {\em International Journal of High Performance Computing
  Applications}, 24(4):511--515, 2010.

\bibitem{fuhry2014discontinuous}
Martin Fuhry, Andrew Giuliani, and Lilia Krivodonova.
\newblock Discontinuous {Galerkin} methods on graphics processing units for
  nonlinear hyperbolic conservation laws.
\newblock {\em International Journal for Numerical Methods in Fluids},
  76(12):982--1003, 2014.

\bibitem{axelGPU2015}
Axel Modave, Amik St-Cyr, and Tim Warburton.
\newblock {GPU} performance analysis of a nodal discontinuous {Galerkin} method
  for acoustic and elastic models.
\newblock {\em Computers \& Geosciences}, 91:64--76, 2016.

\bibitem{medina2014occa}
David~S Medina, Amik St-Cyr, and T~Warburton.
\newblock {OCCA}: A unified approach to multi-threading languages.
\newblock {\em arXiv preprint arXiv:1403.0968}, 2014.

\bibitem{medina2015okl}
David Medina.
\newblock {\em OKL: A Unified Language for Parallel Architectures}.
\newblock PhD thesis, Rice University, 2015.

\bibitem{williams2009roofline}
Samuel Williams, Andrew Waterman, and David Patterson.
\newblock Roofline: an insightful visual performance model for multicore
  architectures.
\newblock {\em Communications of the ACM}, 52(4):65--76, 2009.

\bibitem{warburton2008taming}
Timothy Warburton and Thomas Hagstrom.
\newblock Taming the {CFL} number for discontinuous {Galerkin} methods on
  structured meshes.
\newblock {\em SIAM Journal on Numerical Analysis}, 46(6):3151--3180, 2008.

\bibitem{reyna2014operator}
Matthew~A Reyna and Fengyan Li.
\newblock Operator bounds and time step conditions for the {DG} and central
  {DG} methods.
\newblock {\em Journal of Scientific Computing}, 62(2):532--554, 2015.

\bibitem{xu2014new}
Zhiliang Xu, Xu-Yan Chen, and Yingjie Liu.
\newblock A new {Runge-Kutta} discontinuous {Galerkin} method with conservation
  constraint to improve {CFL} condition for solving conservation laws.
\newblock {\em Journal of computational physics}, 278:348--377, 2014.

\bibitem{derriennic1985multivariate}
Marie-Madeleine Derriennic.
\newblock On multivariate approximation by {Bernstein}-type polynomials.
\newblock {\em Journal of Approximation Theory}, 45(2):155--166, 1985.

\bibitem{warburton2003constants}
T~Warburton and Jan~S Hesthaven.
\newblock On the constants in $hp$-finite element trace inverse inequalities.
\newblock {\em Computer methods in applied mechanics and engineering},
  192(25):2765--2773, 2003.

\bibitem{farouki2003construction}
Rida~T Farouki, Tim~NT Goodman, and Thomas Sauer.
\newblock Construction of orthogonal bases for polynomials in {Bernstein} form
  on triangular and simplex domains.
\newblock {\em Computer Aided Geometric Design}, 20(4):209--230, 2003.

\end{thebibliography}

\end{document}